\def\blt{\hbox to 0pt{\hskip -15pt $\bullet$ \hss}}

\def\C{\mathbbm{C}}
\def\R{\mathbbm{R}}
\def\1{\mathbbm{1}}
\def\ft{\mathcal{F}}
\def\tbar{\overline{T}}
\def\rbar{\overline{R}}

\documentclass [fleqn,12pt]{article}

\usepackage{latexsym}
\usepackage{exscale}
\usepackage{amsmath}
\usepackage{amssymb}
\usepackage{setspace} 
\usepackage{authblk}
\usepackage{subcaption}
\usepackage{enumitem}
\usepackage{xurl}
\usepackage{bm}

\usepackage{caption}
\usepackage{tabu}
\newtabulinestyle{mydashed=on 2.8pt off 2.8pt}
\usepackage{multirow}
\usepackage{accents}

\def\qed{\hfill$\blacksquare$\\} \newenvironment{proof}{\noindent {\bf
Proof.}}{\qed}

\usepackage{float}
\restylefloat{table}

\usepackage{tikz}
\usetikzlibrary{arrows}
\usetikzlibrary{decorations.markings}
\usetikzlibrary{backgrounds}
\usetikzlibrary{math,calc}

\usepackage{pgfplots}

\usepackage{bbm}
\usepackage{geometry}
\newcommand{\sgn}{\mathop{\mathrm{sgn}}}

\newcommand{\inner}[2]{\ensuremath{\langle #1,#2 \rangle}}

\newtheorem{algorithm1}{Algorithm}[section]
\newtheorem{theorem}{Theorem}[section]
\newtheorem{proposition}{Proposition}[section]
\newtheorem{lemma}[theorem]{Lemma}
\newtheorem{corollary}[theorem]{Corollary}
\newtheorem{definition}{Definition}[section]
\newtheorem{observe}{Observation}[section]
\newtheorem{remark1}[observe]{Remark}

\newenvironment{observation}{\begin{observe}}{\end{observe}}
\newenvironment{remark}{\begin{remark1}}{\end{remark1}}

\newlength{\stepwidth} \newlength{\opcountwidth}
\newlength{\descrwidth} \newlength{\totcountwidth}
\newlength{\adaptopcountwidth} \newlength{\adaptdescrwidth}

\newcounter{saveeqn}

\definecolor{green}{rgb}{0.0, 0.6, 0.0}

\title{\bf Generalized prolate spheroidal functions: algorithms and analysis}
\author{Philip Greengard}

\affil{Columbia University, New York, NY, 10027}

\begin{document}

\maketitle
\date
\begin{abstract}
Generalized prolate spheroidal functions (GPSFs) arise naturally in the study of bandlimited functions as the eigenfunctions of a certain truncated Fourier transform. In one dimension, the theory of GPSFs (typically referred to as prolate spheroidal wave functions) has a long history and is fairly complete. Furthermore, more recent work has led to the development of numerical algorithms for their computation and use in applications. In this paper we consider the more general problem, extending the one dimensional analysis and algorithms to the case of arbitrary dimension. Specifically, we introduce algorithms for efficient evaluation of GPSFs and their corresponding eigenvalues, quadrature rules for bandlimited functions, formulae for approximation via GPSF expansion, and various analytical properties of GPSFs. We illustrate the numerical and analytical results with several numerical examples.
\end{abstract}

\tableofcontents

\section{Introduction}
Prolate spheroidal wave functions (PSWFs) are the natural basis for 
representing bandlimited functions on the interval.
Much of the theory and numerical machinery for PSWFs in one dimension 
is fairly complete (see, for example, \cite{xiao, osipov2, wang2017, gopal2023}).
Slepian et al. showed in \cite{bell4} that
the so-called Generalized Prolate Spheroidal Functions (GPSFs) 
are the natural extension of PSWFs in higher dimensions. 
GPSFs are functions $\psi_j:\R^n \rightarrow \C$ satisfying 
\begin{equation}\label{3700}
\lambda_j\psi_j(x)=\int_B \psi_j(t)e^{ic\inner{x}{t}} dt
\end{equation}
for some $\lambda_j \in \C$ where $B$ denotes the unit ball 
in $\R^n$. 
A function $f: \R^n \rightarrow \C$ is referred to as
bandlimited with bandlimit $c>0$ if there exists a square-integrable 
function $\sigma: B \to \R$ such that 
\begin{equation}
f(x)=\int_B \sigma(t) e^{ic\inner{x}{t}}dt.
\end{equation}
Bandlimited functions are encountered
in a variety of applications including in signal processing,
antenna design, radar, etc. 
GPSFs have been used in applications including in cryo-EM \cite{lederman2017}.

Much of the theory and numerical machinery of GPSFs in two 
dimensions is described in \cite{yoel}. 
GPSFs are studied in arbitrary dimension in \cite{zhang2020}, 
including analysis and algorithms
for evaluating GPSFs and their associated eigenvalues. 
The numerical methods of \cite{zhang2020} are an extension of 
the Bouwkamp-type algorithms of \cite{bouwkamp1950} to arbitrary dimensions. 
It is unclear if these algorithms extend stably to the high bandlimit regime. 
This work is to a large extent an extension to arbitrary dimensions of 
the analytical and numerical tools introduced 
in \cite{osipov2} for the $1$-dimensional PSWFs and \cite{yoel} for the 
and $2$-dimensional GPSFs.
We introduce algorithms for evaluating GPSFs, 
quadrature rules for integrating bandlimited functions,
and numerical schemes for representing bandlimited 
functions as GPSF expansions. We also provide numerical machinery
for efficient evaluation of eigenvalues $\lambda_j$ (see (\ref{3700})).
The algorithms of this paper have been implemented in an open source
software package.\footnote{\url{https://github.com/pgree/gpsfexps}}

There are several observations underlying the algorithms and analysis
in this paper. First, the truncated Fourier transform operator, \eqref{3700}, 
when applied to a spherical harmonic expansion separates into a sequence 
of one-dimensional integral operators applied to the radial components of 
the spherical harmonic expansion. 
The eigenfunctions and eigenvalues of those integral operators are 
also eigenfunctions and eigenvalues of the truncated Fourier transform. 
For many of these integral operators, the 
magnitude of the large eigenvalues are identical up to machine precision. 
Thus, standard algorithms  (see, e.g., \cite{yarvin1998}) would be unsuitable for computing their eigendecomponsitions.
Surprisingly, each integral operator commutes with a differential operator
that is tridiagonal and symmetric when acting on a basis of Zernike polynomials.
Furthermore, the Zernike expansion of a GPSF has exponentially decaying 
coefficients. 
These observations can be used to construct numerical algorithms for the 
stable evaluation of the eigendecomposition of the truncated Fourier transform 
operator. 

The family of functions that we consider in this work, GPSFs, concentrate
on the unit ball in both spatial and frequency domains. 
For functions concentrating on hypercubes in both domains, 
eigenfunctions of the operator corresponding to \eqref{3700} are a tensor product 
of PSWFs.
Other work has focused on more general geometries in spatial and frequency 
domains \cite{israel2024, hughes2024, marceca2024}. These 
focus on analysis of the spectra of the spatial-limiting and bandlimiting operators.

The structure of this paper is as follows. In Section \ref{secprem},
we provide basic mathematical background that will be used
throughout the remainder of the paper. 
Section \ref{secanap} contains analytical facts related to the numerical 
evaluation of GPSFs that will be used in subsequent sections. 
In Section \ref{secnumev}, we describe a numerical scheme for 
evaluating GPSFs. 
Section \ref{secquads} contains a quadrature rule for integrating
bandlimited functions. 
Section \ref{secinterp} includes a numerical scheme for expanding
bandlimited functions into GPSF expansions. 
In Section \ref{secnumres}, we provide the numerical results of
implementing the quadrature and approximation schemes
as well as plots of GPSFs and their eigenvalues. 
In the appendix, we include miscellaneous technical lemmas relating to 
GPSFs. 
\section{Mathematical and numerical preliminaries}\label{secprem}
In this section, we introduce notation and elementary
mathematical and numerical facts which will be used 
in subsequent sections.

In accordance with standard practice, we define the
Gamma function, $\Gamma(x)$, by the formula
\begin{equation}\label{1780}
\Gamma(x)=\int_0^\infty t^{x-1}e^{-t} dt
\end{equation}
where $e$ will denote the base of the natural logarithm. 
We will be denoting by 
$\delta_{i,j}$ the function defined by the formula
\begin{equation}\label{40}
\delta_{i,j} = \left\{
  \begin{array}{ll}
  1 & \mbox{if $i = j$}, \\
  0 & \mbox{if $i \ne j$}.
  \end{array}
\right.
\end{equation}
The following is a well-known technical lemma that follows immediately
from 5.6.1 in \cite{dlmf}. It will 
be used in Section \ref{secdecay}.
\begin{lemma}\label{1800}
For any real number $a>0$ and for any integer $n>ae$, 
\begin{equation}
\frac{a^n\sqrt{n}}{\Gamma(n+1)}<1
\end{equation}
where $\Gamma(n)$ is defined in (\ref{1780}). 
\end{lemma}
The following lemma follows immediately from Formula 10.2.2 in \cite{dlmf}.
\begin{lemma}\label{1840}
For all real numbers $x\in [0,1]$, and for all real numbers
$\nu \geq -1/2$,
\begin{equation}
|J_\nu(x)| \leq \frac{|x/2|^{\nu}}{\Gamma(\nu+1)}
\end{equation}
where $J_\nu$ is a Bessel function of the first kind 
and $\Gamma(\nu)$ is defined in (\ref{1780}). 
\end{lemma}
The following proposition provides well-known formulas for the volume 
and area of a $(p+2)$-dimensional hypersphere. The formulas 
can be found in, for example, \cite{li}.

\begin{proposition} \label{thm21.1}
Suppose that $S^{p+1}(r)=\{ x\in \R^{p+2} : \|x\|=r \}$ denotes the 
$(p+1)$-dimensional hypersphere of radius $r>0$. Suppose further 
that $A_{p+2}(r)$ denotes the area
of $S^{p+1}(r)$ and $V_{p+2}(r)$ denotes the volume enclosed by 
$S^{p+1}(r)$. 
Then
\begin{align}\label{3400}
A_{p+2}(r)=\frac{2\pi^{p/2+1}}{\Gamma(\frac{p}{2}+1)}r^{p+1},
\end{align}
and
\begin{align}
V_{p+2}(r)=\frac{\pi^{p/2+1}}{\Gamma(\frac{p}{2}+2)} r^{p+2}.
\label{21.10}
\end{align}

\end{proposition}

\subsection{Spherical harmonics in $\R^{p+2}$}\label{secsh}
Suppose that $S^{p+1}$ denotes the unit sphere in $\R^{p+2}$. 
Spherical harmonics are a set of real-valued continuous functions on
$S^{p+1}$, which are orthonormal and complete in $L^2(S^{p+1})$. The
spherical harmonics of degree $N\ge 0$ are denoted by $S_N^1, S_N^2,
\ldots, \allowbreak S_N^\ell, \ldots, S_N^{h(N, p)}\colon S^{p+1} \to \R$, 
where $h(N, p)$, the degeneracy of degree $N$, is defined by, 
\begin{align}\label{2400}
h(N, p)=(2N+p) \frac{(N+p-1)!} {p!\,N!},
\end{align}
for all nonnegative integers $N, p$.

The following theorem defines the spherical harmonics as the 
values of certain harmonic, homogeneous polynomials on the sphere 
(see, for example,~\cite{bateman2}).

\begin{theorem}\label{1420}
For each spherical harmonic $S_N^\ell$, where $N\ge 0$ and $1\le \ell
\le h(N, p)$ are integers, there
exists a polynomial $K_N^\ell \colon \R^{p+2} \to \R$ which is 
harmonic, i.e.
  \begin{align}
\nabla^2 K_N^\ell(x) = 0,
  \end{align}
for all $x\in \R^{p+2}$, and homogeneous of degree $N$, i.e.,
  \begin{align}
K_N^\ell(\lambda x) = \lambda^N K_N^\ell(x),
  \end{align}
for all $x \in \R^{p+2}$ and $\lambda\in \R$, such that
\begin{align}\label{1400}
S_N^\ell(\xi) = K_N^\ell(\xi),
\end{align}
for all $\xi \in S^{p+1}$.

\end{theorem}
The following lemma follows immediately from the orthonormality
of spherical harmonics and Theorem \ref{1420}.
\begin{lemma}\label{1440}
For all $N>0$ and for all $1 \leq \ell \leq h(N, p)$,
\begin{equation}
\int_{S^{p+1}} S_N^\ell(x) dx = 0.
\end{equation}
For $N=0$ and $\ell=1$, $S_N^\ell$ is the constant function
defined by the formula
\begin{equation}\label{s01}
S_0^1(x)=[A_{p+2}(1)]^{-1/2}
\end{equation}
where $A_{p+2}$ is defined in (\ref{3400}). 
\end{lemma}
The following lemma bounds the $L^{\infty}$ norm of $S^{\ell}_N$. A proof
can be found in, for example, \cite{garrett2013}.
\begin{lemma}
For all $x \in S^{p+1}$, we have 
\begin{align}\label{sh_supnorm}
S^{\ell}_N (x)^2 \leq \frac{h(N, p)}{V_{p+2}} = \frac{\frac{(2N + p)(N + p - 1)!}{p!N!}}{\pi ^{p/2 + 1} / \Gamma(p/2 + 2)}
\end{align}
where $V_{p+2}$ denotes the volume of the unit sphere $S_{p+1}$ in $\R^{p+2}$. 
\end{lemma}
%
%
%
%
%
%
%

%
%
%
%
%
%
%
%
%
\subsection{Jacobi polynomials}\label{secjacpol}
In this section, we summarize some properties Jacobi polynomials.\\
Jacobi polynomials, denoted $P_n^{(\alpha,\beta)}$, are orthogonal 
polynomials on the interval $(-1,1)$ with respect to weight function
\begin{equation}
w(x)=(1-x)^{\alpha} (1+x)^{\beta}.
\end{equation}
Specifically, for all non-negative integers $n,m$ with $n\neq m$ and 
real numbers $\alpha,\beta>-1$,
\begin{equation}\label{192}
\int_{-1}^{1}P_n^{(\alpha,\beta)}(x)P_m^{(\alpha,\beta)}(x)
(1-x)^{\alpha} (1+x)^{\beta}dx=0
\end{equation}
The following lemma, provides a stable recurrence relation that can 
be used to evaluate a particular class of Jacobi polynomials 
(see, for example, \cite{abramowitz}).
\begin{lemma}
For any integer $n\geq1$ and $\alpha>-1$,
\begin{align}\label{195}
&P_{n+1}^{(\alpha,0)}(x)=
\frac{(2n+\alpha+1)\alpha^2+(2n+\alpha)(2n+\alpha+1)(2n+\alpha+2)x}
{2(n+1)(n+\alpha+1)(2n+\alpha)}
P_n^{(\alpha,0)}(x) \notag \\
&-\frac{2n(n+\alpha)(2n+\alpha+2)}
{2(n+1)(n+\alpha+1)(2n+\alpha)}
P_{n-1}^{(\alpha,0)}(x),
\end{align}
where 
\begin{equation}
P_{0}^{(\alpha,0)}(x)=1
\end{equation}
and 
\begin{equation}
P_{1}^{(\alpha,0)}(x)=\frac{\alpha+(\alpha+2)x}{2}.
\end{equation}
The Jacobi polynomial $P_n^{(\alpha,0)}$ is defined in (\ref{192}).
\end{lemma}
%
The following two lemmas, which provide a differential equation 
and a recurrence relation for Jacobi polynomials, 
can be found in, for example, \cite{abramowitz}.
\begin{lemma}\label{205}
For any integer $n \geq 2$ and $\alpha >-1$,
\begin{equation}\label{210.2}
(1-x^2)P_n^{(\alpha,0)\prime\prime}(x)+(-\alpha-(\alpha+2)x)
P_n^{(\alpha,0)\prime}(x)+n(n+\alpha+1)P_n^{(\alpha,0)}(x)=0
\end{equation}
for all $x\in [0,1]$ where $P_n^{(\alpha,0)}$ is defined in (\ref{192}).
\end{lemma}
\begin{lemma}\label{1220a}
For all $\alpha >-1$, $x\in (0,1)$, and any integer $n\geq 2$,
\begin{equation}\label{230}
a_{1n}P_{n+1}^{(\alpha,0)}=(a_{2n}+a_{3n}x)P_n^{(\alpha,0)}(x)
-a_{4n}P_{n-1}^{(\alpha,0)}(x)
\end{equation}
where 
\begin{equation}\label{240}
\begin{aligned}
a_{1n}&=2(n+1)(n+\alpha+1)(2n+\alpha)\\
a_{2n}&=(2n+\alpha+1)\alpha^2\\
a_{3n}&=(2n+\alpha)(2n+\alpha+1)(2n+\alpha+2)\\
a_{4n}&=2n(n+\alpha)(2n+\alpha+2)
\end{aligned}
\end{equation}
and
\begin{equation}
\begin{split}
&P_{0}^{(\alpha,0)}(x)=1\\
&P_{1}^{(\alpha,0)}(x)=\frac{\alpha+(\alpha+2)x}{2}.
\end{split}
\end{equation}
\end{lemma}
\subsection{Zernike polynomials}
In this section, we describe properties of Zernike polynomials, 
which are a family of orthogonal polynomials on the unit ball in
$\R^{p+2}$. They are the natural basis for representing GPSFs. 
A more in-depth description of Zernike polynomials, including
their properties and relevant numerical algorithms can be found 
in \cite{pkzern}. In accordance with \cite{bell4}, for the remainder 
of this paper we the denote the dimension by $p+2$.
Zernike polynomials are defined via the formula
\begin{equation}\label{6.10}
Z_{N,n}^\ell(x) = R_{N,n}(\|x\|) S_N^\ell(x/\|x\|),
\end{equation}    
for all $x\in \R^{p+2}$ such that $\|x\| \le 1$, 
where $N$ and $n$ are nonnegative
integers, $S_N^\ell: S^{p + 1} \to \R$, $\ell = 1,...,h(N, p)$,
are the orthonormal spherical harmonics of 
degree $N$ (see Section \ref{secsh}), and $R_{N,n}$ are polynomials 
of degree $2n+N$ defined via the formula
\begin{equation}\label{6.20}
R_{N,n}(r) = r^N \sum_{m=0}^n (-1)^m {n+N+\frac{p}{2} \choose m} 
  {n\choose m} (r^2)^{n-m} (1-r^2)^m,
\end{equation}
for all $0\le r\le 1$. 
Here, $n$ denotes the radial order of the Zernike polynomials.
We note that the spherical harmonics $S_{N}^{\ell}: S^{p +1} \to \R$ depend 
implicitly on the dimension, $p +1$. 
The polynomials $R_{N,n}$ satisfy the relation
\begin{equation}\label{6.30}
R_{N,n}(1) = 1,
\end{equation}
and are orthogonal with respect to the weight function $w(r) = r^{p+1}$, so
that
\begin{equation}\label{6.40}
\int_0^1 R_{N,n}(r)R_{N,m}(r) r^{p+1}\, dr = \frac{\delta_{n,m}}
  {2(2n+N+\frac{p}{2}+1)}.
\end{equation}
We define the polynomials $\overline{R}_{N,n}$ via the formula
\begin{equation}\label{6.60}
\overline{R}_{N,n}(r) = \sqrt{2(2n+N+p/2+1)} R_{N,n}(r),
\end{equation}
so that
\begin{equation}\label{6.70}
\int_0^1 \bigl(\overline{R}_{N,n}(r)\bigr)^2 r^{p+1}\, dx = 1,
\end{equation}
where $N$ and $n$ are nonnegative integers.
In an abuse of notation, we refer to both 
the polynomials $Z^\ell_{N,n}$ and the radial polynomials
$R_{N,n}$  as Zernike polynomials where the meaning is obvious.

\begin{remark}
When $p=-1$, Zernike polynomials take the form
\begin{equation}\label{6.100}
\begin{split}
&Z_{0,n}^1(x) = R_{0,n}(|x|) = P_{2n}(x), \\
&Z_{1,n}^2(x) = \sgn(x)\cdot R_{1,n}(|x|) = P_{2n+1}(x), 
\end{split}
\end{equation}
for $-1\le x\le 1$ and nonnegative integer $n$, where $P_n$ denotes the
Legendre polynomial of degree $n$ and
  \begin{align}
  \sgn(x) = \left\{
    \begin{array}{ll}
    1  & \mbox{if $x > 0$}, \\
    0  & \mbox{if $x = 0$}, \\
    -1 & \mbox{if $x < 0$},
    \end{array}
    \right.
    \label{6.110}
  \end{align}
for all real $x$.
\end{remark}

\begin{remark} \label{rem6.1}
When $p=0$, Zernike polynomials take the form
  \begin{align}
&Z_{N,n}^1(x_1,x_2) 
  = R_{N,n}(r) \cos(N\theta), \label{6.71}  \\
&Z_{N,n}^2(x_1,x_2)  
  = R_{N,n}(r) \sin(N\theta), \label{6.80}
  \end{align}
where $x_1=r\cos(\theta)$, $x_2=r\sin(\theta)$, and $N$ and $n$ are
nonnegative integers. 
In this case, the spherical harmonics, $S_{N}^{\ell}$, of \eqref{6.10}
are $\sin(N\theta), \cos(N\theta)$ where $\ell$ distinguishes $\sin$ 
from $\cos$.
The term Zernike polynomials is often used to refer to the 
Zernike polynomials defined on the unit disk in $\R^2$, that is, with $p=0$.
This is largely due to their use in optics communities \cite{born1999}.
What we refer to as Zernike polynomials in
this paper are their generalization to arbitrary dimension. 
\end{remark}
The following lemma, which can be found in, for example, 
\cite{abramowitz}, shows how Zernike polynomials are related 
to Jacobi polynomials. 
\begin{lemma}\label{1040}
For all non-negative integers $N,n$, 
\begin{equation}\label{6.240.2}
R_{N,n}(r)=(-1)^n r^N P_n^{(N+\frac{p}{2},0)}(1-2r^2),
\end{equation}
where $0\le r\le 1$, and $P^{(\alpha,0)}_n$, $\alpha > -1$, is 
defined in (\ref{6.20}).
\end{lemma}
\subsection{Numerical evaluation of Zernike polynomials}\label{secnumev_zern}
In this section, we provide a stable recurrence relation 
(see Lemma \ref{lem755}) that can be used to 
evaluate Zernike polynomials.

The following lemma follows immediately from applying 
Lemma \ref{1040} to (\ref{230}).
\begin{lemma}\label{lem755}
The polynomials $R_{N,n}$, defined in (\ref{6.20}) satisfy the 
recurrence relation
\begin{align}\label{7550}
&\hspace*{-1em}
R_{N,n+1}(r)= \notag \\
&\hspace*{-1em}
-\frac{((2n+N+1)N^2+(2n+N)(2n+N+1)(2n+N+2)(1-2r^2))}
{2(n+1)(n+N+1)(2n+N)}
R_{N,n}(r)  \notag \\
&\hspace*{-1em}
-\frac{2n(n+N)(2n+N+2)}
{2(n+1)(n+N+1)(2n+N)}
R_{N,n-1}(r)
\end{align}
where $0\le r\le 1$, $N$ is a non-negative integer, $n$ is a 
positive integer, and
\begin{equation}
R_{N,0}(r)=r^N
\end{equation}
and 
\begin{equation}
R_{N,1}(r)=-r^N\frac{N+(N+2)(1-2r^2)}{2}.
\end{equation}
\end{lemma}
\begin{remark}
The algorithm for evaluating Zernike polynomials using the recurrence
relation in Lemma~\ref{lem755} is known as Kintner's method
(see~\cite{kintner} and, for example,~\cite{chong}).
\end{remark}
\subsection{Modified Zernike polynomials}
In this section, we define the modified Zernike polynomials,
$\tbar_{N,n}$ and provide some of their properties. 
This family of functions will be used in Section \ref{secnumev}
for the numerical evaluation of GPSFs. 

We define the function $T_{N,n}$ by the formula 
\begin{equation}\label{3600}
T_{N,n}(r)=r^{\frac{p+1}{2}}R_{N,n}(r)
\end{equation}
where $N,n$ are non-negative integers. 
We define $\tbar_{N,n}: [0,1] \rightarrow \R$ by the formula,
\begin{equation}\label{400}
\tbar_{N,n}(r)=r^{\frac{p+1}{2}}\rbar_{N,n}(r)
\end{equation}
where $N,n$ are non-negative integers and $\rbar_{N,n}$ is a
normalized Zernike polynomial defined in (\ref{6.60}), so that
\begin{equation}
\int_0^1 (\tbar_{N,n}(r))^2 dr =1.
\end{equation}
\begin{lemma}\label{500}
The functions $\tbar_{N,n}$ are orthonormal on the interval $(0,1)$ 
with respect to weight function $w(x)=1$. That is,
\begin{equation}
\int_0^1 \tbar_{N,n}(r) \tbar_{N,m}(r) dr = \delta_{n,m}.
\end{equation}
\end{lemma}
\begin{proof}
Using (\ref{400}), (\ref{6.40}) and (\ref{6.70}), for all 
non-negative integers $N,n,m$,
\begin{equation}
\begin{split}
\int_0^1 \tbar_{N,n}(r) \tbar_{N,m}(r) dr 
&= \int_0^1 r^{\frac{p+1}{2}}\rbar_{N,n}(r) 
r^{\frac{p+1}{2}}\rbar_{N,m}(r) dr \\
&= \int_0^1 \rbar_{N,n}(r) \rbar_{N,m}(r) r^{p+1} dr \\
&=\delta_{n,m}
\end{split}
\end{equation}
\end{proof}
The following identity follows immediately from the combination 
of (\ref{400}),(\ref{6.240.2}), and (\ref{6.60}). 
\begin{lemma}\label{460}
For all non-negative integers $N,n$,
\begin{equation}\label{260}
\begin{aligned}
\tbar_{N,n}(r)=P_n^{(N+p/2,0)}(1-2r^2)
r^{N} (-1)^n \sqrt{2(2n+N+p/2+1)}r^{ \frac{p+1}{2}}
\end{aligned}
\end{equation}
where $\tbar_{N,n}$ is defined in (\ref{400}) and 
$P_n^{(N+p/2,0)}$ is a Jacobi polynomial defined in (\ref{192}). 
\end{lemma}
The following observation provides a numerical scheme for computing $\tbar_{N,n}$
and follows immediately from combining Lemma \ref{460} with 
\eqref{230}.
\begin{observation}\label{2000}
The modified Zernike polynomials $\tbar_{N, n}$ can be evaluated
by first computing
\begin{equation}
P_n^{(N+p/2,0)}(1-2r^2)
\end{equation}
via recurrence relation \eqref{230} and then multiplying the resulting 
number by 
\begin{equation}
r^N (-1)^n \sqrt{2(2n+N+p/2+1)} r^{\frac{p+1}{2}}.
\end{equation}
\end{observation}

The following lemma, which provides a differential equation for 
$\tbar_{N,n}$, follows immediately from substituting (\ref{260})
into \eqref{210.2}.
\begin{lemma}\label{lem5.20}
For all $r \in [0,1]$, non-negative integers $N,n$ and real
$p\geq -1$, 
\begin{equation}\label{200}
(1-r^2)\tbar''_{N,n}(r)-2r\tbar'_{N,n}(r)+\left(\chi_{N,n} 
+ \frac{\frac{1}{4} - (N+\frac{p}{2})^2}{r^2}\right)\tbar_{N,n}(r) = 0
\end{equation}
where $\chi_{N,n}$ is defined by the formula
\begin{equation}\label{210}
\chi_{N,n}=(N+p/2+2n+1/2)(N+p/2+2n+3/2).
\end{equation}
\end{lemma}
The following lemma provides a recurrence relation satisfied 
by $\tbar_{N,n}$. It follows immediately from the combination 
of Lemma \ref{460} and \eqref{195}.
\begin{lemma}\label{860}
For any non-negative integers $N,n$ and for all $r\in [0,1]$,
\begin{equation}\label{270}
\begin{aligned}
r^2\tbar_{N,n}(r)=&\frac{\sqrt{2(2n+N+p/2+1)}}{\sqrt{2(2(n-1)+N+p/2+1)}}
\frac{a_{4n}}{2a_{3n}}\tbar_{N,n-1}(r)\\
&+\frac{a_{2n}+a_{3n}}{2a_{3n}}\tbar_{N,n}(r)\\
&+\frac{\sqrt{2(2n+N+p/2+1)}}{\sqrt{2(2(n+1)+N+p/2+1)}}
\frac{a_{1n}}{2a_{3n}}\tbar_{N,n+1}(r)
\end{aligned}
\end{equation}
where $\tbar_{N,n}$ is defined in (\ref{400}) and 
\begin{equation}
\begin{aligned}
a_{1n}&=2(n+1)(n+N+p/2+1)(2n+N+p/2)\\
a_{2n}&=(2n+N+p/2+1)(N+p/2)^2\\
a_{3n}&=(2n+N+p/2)(2n+N+p/2+1)(2n+N+p/2+2)\\
a_{4n}&=2n(n+N+p/2)(2n+N+p/2+2).
\end{aligned}
\end{equation}
\end{lemma}
\begin{proof}
Applying the change of variables $1-2r^2=x$ to (\ref{230}) and 
setting $\alpha=N+p/2$, we obtain
\begin{equation}\label{250}
\begin{aligned}
r^2P_n^{(N+p/2,0)}(1-2r^2)=&\frac{a_{2n}}{2a_3n}P^{(N+p/2,0)}_n(1-2r^2)
+\frac{1}{2}P^{(N+p/2,0)}_n(1-2r^2)\\
&-\frac{a_{4n}}{2a_3n}P^{(N+p/2,0)}_{n-1}(1-2r^2)
-\frac{a_{1n}}{2a_3n}P^{(N+p/2,0)}_{n+1}(1-2r^2).
\end{aligned}
\end{equation}
Identity (\ref{270}) follows immediately from 
the combination of (\ref{250}) with Lemma \ref{460}.
\end{proof}
We define the function $\tilde{T}_{N, n}$ by the formula
\begin{equation}\label{600}
\tilde{T}_{N, n}(r)=\frac{\tbar_{N,n}(r)}{r^{N+\frac{p+1}{2}}}.
\end{equation}
where $N,n$ are non-negative integers and $r \in(0,1)$.
We use $\tilde{T}_{N, n}$ in Section \ref{sec:one_eig} in the proof 
of Theorem \ref{960}, which provides a formula for 
efficiently evaluating a single eigenvalue  $\beta_{N, n}$ of 
the prolate integral operator (see \eqref{5.100}).  
The following technical lemma involving $\tilde{T}_{N, n}$
will be used in Section \ref{sectridiag}.
\begin{lemma}\label{820}
For all non-negative integers $N,n$,
\begin{equation}\label{60}
\tilde{T}_{N, n}(0)= \sqrt{2(2n+N+p/2+1)}(-1)^n\dbinom{n+N+p/2}{n}.
\end{equation}
\end{lemma}
\begin{proof}
Combining (\ref{400}) and (\ref{6.20}), we observe that
\begin{equation}\label{540}
\tbar_{N,n}(r)=\sum_{k=0}^{n} a_{N+k}r^{N+\frac{p+1}{2}+2k}
\end{equation}
where $a_{N+k}$ is some real number for $k=0,1,...,n$.
In particular, using (\ref{6.20}),
\begin{equation}\label{560}
a_N=\sqrt{2(2n+N+p/2+1)}(-1)^n\dbinom{n+N+p/2}{n}.
\end{equation}
Combining (\ref{600}) and (\ref{560}), we obtain (\ref{60}).
\end{proof}
The following lemma follows immediately from (\ref{540}) and 
provides a relation that will be used in 
Section \ref{sec:all_eigs} for the evaluation of certain eigenvalues. 
\begin{lemma}\label{2920}
Suppose that $N$ is a nonnegative integer and that $n\ge 1$ is an integer.
Then
\begin{equation}\label{2940}
\begin{split}
&\tilde a_n r T_{N,n-1}'(r) - \tilde b_n r T_{N,n}'(r)
+ \tilde c_n r T_{N,n+1}'(r) \\
&= a_n T_{N,n-1}(r) - b_n T_{N,n}(r) + c_n T_{N,n+1}(r),
\end{split}
\end{equation}
for all $0\le r\le 1$, where
\begin{equation}
\begin{split}
&\tilde a_n = -2n(2n+N+p/2+2), \\
&\tilde b_n = 2(N + p/2)(2n+N + p/2+1), \\
&\tilde c_n = 2(n+N + p/2+1)(2n+N + p/2), \\
& a_n = n(2(N + p/2)+4n-1)(2n+N + p/2+2), \\
& b_n = (N + p/2)(2n+N + p/2+1)-2(2n+N + p/2)_3, \\
& c_n = (2(N + p/2)+4n+5)(n+N + p/2+1)(2n+N + p/2), \\
\end{split}
\end{equation}
with $(\cdot)_k$ denoting the Pochhammer symbol or rising factorial.
\end{lemma}
\subsection{Pr{\"u}fer transform}\label{secprufer}
In this section, we describe the Pr{\"u}fer Transform, 
which will be used in Section \ref{secroots} in an algorithm
for finding the roots of GPSFs. A more detailed 
description of the Pr{\"u}fer Transform can be found in \cite{glaser}. 

\begin{lemma}[Pr{\"u}fer Transform]\label{380.2}
Suppose that the function $\phi: [a,b] \rightarrow \R$ 
satisfies the differential equation
\begin{equation}\label{400.2}
\phi^{\prime\prime}(x)+\alpha(x)\phi^{\prime}(x)
+\beta(x)\phi(x)=0,
\end{equation}
where $\alpha,\beta:(a,b) \rightarrow \R$ are differentiable functions.
Then,
\begin{equation}\label{440.2}
\frac{d\theta}{dx}=-\sqrt{\beta(x)}-\left(\frac{\beta^{\prime}(x)}
{4\beta(x)}+\frac{\alpha(x)}{2}\right)\sin(2\theta),
\end{equation}
where the function $\theta :[a,b]\rightarrow \R$ is defined 
by the formula,
\begin{equation}\label{420.2}
\frac{\phi^\prime(x)}{\phi(x)}=\sqrt{\beta(x)}\tan(\theta(x)).
\end{equation}
\end{lemma}
\begin{proof}
Introducing the notation
\begin{equation}\label{460.2}
z(x)=\frac{\phi^{\prime}(x)}{\phi(x)}
\end{equation}
for all $x\in [a,b]$, and 
%
%
substituting (\ref{460.2}) into (\ref{400.2}),
we obtain,
\begin{equation}\label{500.2}
\frac{dz}{dx}=-(z^2(x)+\alpha(x)z(x)+\beta(x)).
\end{equation}
Introducing the notation,
\begin{equation}\label{520.2}
z(x)=\sqrt{\beta(x)}\tan(\theta(x)),
\end{equation}
with $\theta$ an unknown function, 
we differentiate (\ref{520.2}) and observe that,
\begin{equation}\label{540.2}
\frac{dz}{dx}=\sqrt{\beta(x)}\frac{\theta^{\prime}}{\cos^2(\theta)}
+\frac{\beta^{\prime}(x)}{2\beta(x)}\tan(\theta(x)).
\end{equation}
%
%
%
Substituting (\ref{520.2}) and (\ref{540.2}) into 
(\ref{500.2}) we obtain
\begin{equation}\label{600.2}
\frac{d\theta}{dx}=-\sqrt{\beta(x)}-\left(\frac{\beta^{\prime}(x)}
{4\beta(x)}+\frac{\alpha(x)}{2}\right)\sin(2\theta).
\end{equation}
\end{proof}
\begin{remark}\label{620.2}
The Pr{\"u}fer Transform is often used in algorithms for finding
the roots of oscillatory special functions. For instance, 
suppose that $\phi: [a,b] \rightarrow \R$ is a special 
function satisfying differential 
equation (\ref{400.2}). It turns out that in most cases, coefficient
$\sqrt{\beta(x)}$ in (\ref{400.2}) is significantly larger than 
\begin{equation}\label{660.2}
\frac{\beta^{\prime}(x)}{4\beta(x)}+\frac{\alpha(x)}{2}
\end{equation}
on the interval $[a,b]$, where $\alpha$ and $\beta$ 
are defined in (\ref{400.2}).

Under these conditions, the function $\theta$ (see (\ref{420.2})), 
is monotone and its derivative neither approaches infinity nor $0$.
Furthermore, finding the roots of $\phi$ is equivalent to 
finding $x \in [a,b]$ such that
\begin{equation}\label{670.2}
\theta(x)=\pi/2+k\pi
\end{equation}
for some integer $k$.
\end{remark}
\begin{remark}\label{680.2}
If for all $x\in [a,b]$, the function $\sqrt{\beta(x)}$ satisfies
\begin{equation}\label{700.2}
\sqrt{\beta(x)}>
\frac{\beta^{\prime}(x)}
{4\beta(x)}+\frac{\alpha(x)}{2},
\end{equation}
then, for all $x\in [a,b]$, we have 
$\frac{d\theta}{dx} < 0$ (see (\ref{440.2}))
and we can view $x:[-\pi / 2,\pi / 2]\rightarrow \R$
as a function of $\theta$ where $x$ satisfies the first order
differential equation
\begin{equation}\label{720.2}
\frac{dx}{d\theta}=\left(-\sqrt{\beta(x)}-\left(\frac{\beta^{\prime}(x)}
{4\beta(x)}+\frac{\alpha(x)}{2}\right)\sin(2\theta)\right)^{-1}.
\end{equation}
\end{remark}

\subsection{Generalized prolate spheroidal functions}

\subsubsection{Basic facts} \label{sect4}

In this section, we summarize several facts about generalized 
prolate spheroidal functions (GPSFs). 
Let $B$ denote the closed unit ball in $\R^{p+2}$. Given a real number
$c>0$, we define the operator $F_c\colon L^2(B) \to
L^2(B)$ via the formula
  \begin{equation}
F_c[\psi](x) = \int_B \psi(t) e^{ic\inner{x}{t}}\, dt,
    \label{4.10}
  \end{equation}
for all $x\in B$, where $\inner{\cdot}{\cdot}$ denotes the inner product on
$\R^{p+2}$. The operator $F_c$ is compact because its kernel is continuous
\cite{kress2013}. 
$F_c$ is also normal, but not
self-adjoint.  We denote the eigenvalues of $F_c$ by $\lambda_0, \lambda_1,
\ldots, \lambda_n, \ldots$, and assume that $|\lambda_j| \ge |\lambda_{j+1}|$
for each non-negative integer $j$.  For each non-negative integer $j$, we
denote by $\psi_j$ the eigenfunction corresponding to $\lambda_j$, so that
  \begin{equation}\label{4.20}
\lambda_j\psi_j(x) = \int_B \psi_j(t) e^{ic\inner{x}{t}}\, dt,
  \end{equation}
for all $x \in B$. We assume that $\|\psi_j\|_{L^2(B)}=1$ for each $j$.
The following theorem is proved in~\cite{bell4} and describes the
eigenfunctions and eigenvalues of $F_c$.

\begin{theorem}\label{thm4.1}
Suppose that $c > 0$ is a real number and that $F_c$ is defined by~{\rm
(\ref{4.10})}.  Then the eigenfunctions $\psi_0,\psi_1,\ldots,
\psi_n,\ldots$ of $F_c$ are real, orthonormal, and complete in $L^2(B)$.  For
each $j$, the eigenfunction $\psi_j$ is either even, in the sense that
$\psi_j(-x) = \psi_j(x)$ for all $x\in B$,  or odd, in the sense that $\psi_j(-x) =
-\psi_j(x)$ for all $x\in B$. The eigenvalues corresponding to even
eigenfunctions are real, and the eigenvalues corresponding to odd
eigenfunctions are purely imaginary.  The domain on which the eigenfunctions
are defined can be extended from $B$ to $\R^{p+2}$ by requiring that ~{\rm
(\ref{4.20})} hold for all $x \in \R^{p+2}$; the eigenfunctions will then be
orthogonal on $\R^{p+2}$ and complete in the class of band-limited functions with
bandlimit $c$.
\end{theorem}

\begin{remark}\label{psi_cn}
The function $\psi_j$ admits an analytic extension to $\R^n$ defined
by equation \eqref{4.20}. In a slight abuse of notation we denote both 
the function defined on $B$ and its analytic extension to $\R^{p+2}$ by $\psi_j(x)$. 
\end{remark}

We define the self-adjoint operator $Q_c\colon L^2(B) \to L^2(B)$ via the
formula
  \begin{equation}
Q_c = \Bigl(\frac{c}{2\pi}\Bigr)^{p+2} F_c^* \cdot F_c.
    \label{4.30}
  \end{equation}
Since $F_c$ is normal, it follows that $Q_c$ has the same eigenfunctions as
$F_c$, and that the $j$th eigenvalue $\mu_j$ of $Q_c$ is connected to
$\lambda_j$ via the formula
  \begin{equation}
\mu_j = \Bigl(\frac{c}{2\pi}\Bigr)^{p+2} |\lambda_j|^2.
    \label{4.40}
  \end{equation}
We also observe that
  \begin{align}
&Q_c[\psi](x) 
= \Bigl(\frac{c}{2\pi}\Bigr)^{p/2+1}
  \int_B \frac{J_{p/2+1}
  \bigl(c\|x-t\|\bigr)}{\|x-t\|^{p/2+1}} \psi(t) \, dt,
    \label{4.50}
  \end{align}
for all $x \in \R^{p+2}$, where $J_\nu$
denotes the Bessel functions of the first kind and $\|\cdot\|$ denotes
Euclidean distance in $\R^{p+2}$ (see Appendix A for a proof).

We observe that
  \begin{equation}
Q_c[\psi](x) = \1_B(x) \cdot \ft^{-1} \bigl[\1_{B(c)}(t) \cdot
\ft[\psi](t)\bigr](x),
    \label{4.70}
  \end{equation}
where $\ft\colon L^2(\R^{p+2}) \to L^2(\R^{p+2})$ is the
$(p+2)$-dimensional Fourier transform, 
$B(c)$ denotes the set $\{\,x \in \R^{p+2} : \|x\| \le c\,\}$, 
and $\1_A$ is defined via the
formula
  \begin{equation}
  \1_A(x) = \left\{
    \begin{array}{ll}
    1 & \mbox{if $x \in A$}, \\
    0 & \mbox{if $x \not\in A$}.
    \end{array}
  \right.
  \end{equation}

From~(\ref{4.70}) it follows that $\mu_j < 1$ for all
$j$. 

We observe further that $Q_c$ is closely related to the operator $P_c
\colon L^2(\R^{p+2}) \to L^2(\R^{p+2})$, defined via the formula
  \begin{align}
&P_c[\psi](x) 
= \Bigl(\frac{c}{2\pi}\Bigr)^{p/2+1}
  \int_{\R^{p+2}} \frac{J_{p/2+1}
  \bigl(c\|x-t\|\bigr)}{\|x-t\|^{p/2+1}} \psi(t) \, dt,
    \label{4.90}
  \end{align}
which is the orthogonal projection onto the space of bandlimited functions on
$\R^{p+2}$ with bandlimit $c > 0$.
\subsubsection{Eigenfunctions and eigenvalues of $F_c$} \label{sect5}
In this section we describe the eigenvectors and eigenvalues of the
operator $F_c$, defined in~(\ref{4.10}). 
Suppose that $\psi$ is some eigenfunction of the integral operator $F_c$,
with corresponding complex eigenvalue
$\lambda$, so that
  \begin{equation}
\lambda\psi(x) = \int_B \psi(t) e^{ic\inner{x}{t}}\, dt,
    \label{5.10}
  \end{equation}
for all $x\in B$ (see Theorem~\ref{thm4.1}). 
\begin{observation} \label{obs5.1}
The operator $F_c$, defined by~{\rm(\ref{4.10})}, is spherically symmetric
in the sense that, for any
$(p+2)\times (p+2)$ orthogonal matrix $U$,
$F_c$ commutes with the operator $\hat{U}\colon L^2(B) \to L^2(B)$,
defined via the formula
  \begin{equation}
\hat{U}[\psi](x) = \psi(Ux),
    \label{5.20}
  \end{equation}
for all $x\in B$. Hence,
the problem of finding the eigenfunctions and eigenvalues of $F_c$ 
is amenable to the separation of variables.
\end{observation}
This separation implies that the eigenfunctions may be expressed 
as
  \begin{equation}
\psi(x) = \Phi_N^\ell(\|x\|) S_N^\ell(x/\|x\|),
    \label{5.30}
  \end{equation}
where $S_N^\ell$, $\ell=0,1,\ldots,h(N,p)$ denotes the spherical
harmonics of degree $N$ (see Section \ref{secsh}), and $\Phi_N^\ell(r)$ 
is a real-valued function defined on the interval $[0,1]$.  
We observe that
  \begin{align}
&\hspace*{-2em} e^{ic\inner{x}{t}} = \sum_{N=0}^\infty \sum_{\ell=1}^{h(N,p)} 
i^N (2\pi)^{p/2+1} \frac{J_{N+p/2}(c\|x\| \|t\|)}{(c\|x\| \|t\|)^{p/2}}
S_N^\ell(x/\|x\|) S_N^\ell(t/\|t\|),
    \label{5.40}
  \end{align}
where $x,t\in B$, and where $J_\nu$
denotes the Bessel functions of the first kind (see Section~VII
of~\cite{bell4} for a proof). Equation \eqref{5.40} can be viewed as a 
Jacobi-Anger expansion in arbitrary dimension \cite{colton1997}. 
Substituting~(\ref{5.30}) and~(\ref{5.40}) into~(\ref{5.10}), we find that
  \begin{equation}
\lambda \Phi_N^\ell(r) = i^N (2\pi)^{p/2+1}
\int_0^1 \frac{J_{N+p/2}(cr\rho)}{(cr\rho)^{p/2}}
  \Phi_N^\ell(\rho) \rho^{p+1}\, d\rho,
    \label{5.80}
  \end{equation}
for all $0\le r\le 1$. We define the operator $H_{N,c}\colon
L^2\bigl([0,1],\rho^{p+1}\,d\rho\bigr) \to
L^2\bigl([0,1],\rho^{p+1}\,d\rho\bigr)$ via the formula
  \begin{equation}
H_{N,c}[\Phi](r) 
= \int_0^1 \frac{J_{N+p/2}(cr\rho)}{(cr\rho)^{p/2}} \Phi(\rho)
  \rho^{p+1}\, d\rho,
    \label{5.90}
  \end{equation}
where $0\le r\le 1$, and observe that $H_{N,c}$ is self-adjoint and compact 
(its kernel is continuous), and does not depend on $\ell$.  Dropping the index $\ell$, we
denote by $\beta_{N,0},\beta_{N,1}, \ldots,\beta_{N,n}, \ldots$ the
eigenvalues of $H_{N,c}$, and assume that $|\beta_{N,n}| \ge |\beta_{N,n+1}|$
for each nonnegative integer $n$. For each nonnegative integer $n$, we let
$\Phi_{N,n}$ denote the eigenfunction corresponding to eigenvalue
$\beta_{N,n}$, so that 
  \begin{equation}
\beta_{N,n} \Phi_{N,n}(r) = 
\int_0^1 \frac{J_{N+p/2}(cr\rho)}{(cr\rho)^{p/2}}
  \Phi_{N,n}(\rho) \rho^{p+1}\, d\rho,
    \label{5.100}
  \end{equation}
for all $0\le r\le 1$. The eigenfunctions $\Phi_{N,n}$ are purely
real, and we assume that $\|\Phi_{N,n}\|_{L^2([0,1],\rho^{p+1}\, d\rho)} = 1$ and
that $\Phi_{N,n}(1) > 0$ for each nonnegative integer $N$ and $n$ (see
Theorem~\ref{thm12.6}). It follows from~(\ref{5.100}) and~(\ref{5.80}) that the
eigenfunctions and eigenvalues of $F_c$ are given by the formulas
  \begin{equation}
\psi_{N,n}^\ell(x) = \Phi_{N,n}(\|x\|)S_N^\ell(x/\|x\|),
    \label{5.110}
  \end{equation}
and
  \begin{equation}
\lambda_{N,n}^\ell = i^N (2\pi)^{p/2+1} \beta_{N,n},
    \label{5.120}
  \end{equation}
respectively, where $x\in B$, $N$ and $n$ are nonnegative integers, and
$\ell$ is an integer so that $1 \le \ell \le h(N,p)$ 
(see Section~\ref{secsh}).
We note in formula~(\ref{5.120}) the expected degeneracy of eigenvalues
due to the spherical symmetry of the integral operator $F_c$ (see
Observation~\ref{obs5.1}); we denote $\lambda_{N,n}^\ell$ by
$\lambda_{N,n}$  where the meaning is clear.  
  \begin{observation} \label{obs5.2}
The domain on which the functions $\Phi_{N,n}$ are defined may be extended
from the interval $[0,1]$ to the complex plane $\C$ by requiring that~{\rm
(\ref{5.10})} hold for all $r\in\C$. Moreover, the functions $\Phi_{N,n}$,
extended in this way, are entire.
  \end{observation}

\subsubsection{The dual nature of GPSFs} \label{sect7}

In this section, we observe that the eigenfunctions 
$\Phi_{N,0},\Phi_{N,1}, \ldots,\Phi_{N,n}, \ldots$ 
of the integral operator $H_{N,c}$, defined in~(\ref{5.90}), 
are also the eigenfunctions of a certain differential operator.

Let $\beta_{N,n}$ denote the eigenvalue corresponding to
the eigenfunction $\Phi_{N,n}$, for all nonnegative integers $N$ and
$n$, so that 
  \begin{equation}
\beta_{N,n} \Phi_{N,n}(r) = 
\int_0^1 \frac{J_{N+p/2}(cr\rho)}{(cr\rho)^{p/2}}
  \Phi_{N,n}(\rho) \rho^{p+1}\, d\rho,
    \label{7.20}
  \end{equation}
where $0\le r\le 1$, $N$ and $n$ are nonnegative integers, and $J_\nu$
denotes the Bessel functions of the
first kind (see~(\ref{5.100})). Making the substitutions
\begin{equation}\label{7.30}
\varphi_{N,n}(r) = r^{(p+1)/2} \Phi_{N,n}(r),
\end{equation}
and
\begin{equation}\label{7.40}
\gamma_{N,n} = c^{(p+1)/2} \beta_{N,n},
\end{equation}
we observe that
  \begin{equation}
\gamma_{N,n} \varphi_{N,n}(r) = \int_0^1 J_{N+p/2}(cr\rho)\sqrt{cr\rho}\,
  \varphi_{N,n}(\rho)\, d\rho,
    \label{7.50}
  \end{equation}
where $0\le r\le 1$, and $N$ and $n$ are arbitrary nonnegative integers.  We
define the operator $M_{N,c}\colon L^2([0,1]) \to L^2([0,1])$ via the formula
  \begin{equation}
M_{N,c}[\varphi](r) = \int_0^1 J_{N+p/2}(cr\rho)\sqrt{cr\rho}\,
  \varphi(\rho)\, d\rho,
    \label{7.51}
  \end{equation}
where $0\le r\le 1$, and $N$ is an arbitrary nonnegative integer.  $M_{N, n}$ is 
self-adjoint and since its kernel is continuous, $M_{N,c}$ is compact. 
The eigenvalues of $M_{N,c}$
are $\gamma_{N,0}, \gamma_{N,1}, \ldots, \gamma_{N,n}, \ldots$, and
$\varphi_{N,n}$ is the eigenfunction corresponding to eigenvalue
$\gamma_{N,n}$, for each nonnegative integer $n$.

We define the differential operator $L_{N,c}$ via the formula
  \begin{equation}
L_{N,c}[\varphi](x) = \frac{d}{dx} \biggl( (1-x^2) \frac{d\varphi}{dx}(x)
\biggr) + \biggl( \frac{\frac{1}{4} - (N+\frac{p}{2})^2}{x^2}
  - c^2 x^2 \biggr)\varphi(x),
    \label{7.60}
  \end{equation}
where $0< x< 1$, $N$ is a nonnegative integer, and $\varphi$ is twice
continuously differentiable. Let $C$ be the class of functions $\varphi$
which are bounded and twice continuously differentiable on the interval
$(0,1)$, such that $\varphi'(0)=0$ if $p=-1$ and $N=0$, and $\varphi(0)=0$
otherwise. Then it is easy to show that,
operating on functions in class $C$, $L_{N,c}$ is self-adjoint. From 
Sturmian theory we obtain the following theorem (see~\cite{bell4}).

\begin{theorem} \label{thm7.1}
Suppose that $c > 0$, $N$ is a nonnegative integer, and $L_{N,c}$ is
defined via~{\rm(\ref{7.60})}.  Then there exists a strictly increasing
unbounded sequence of positive numbers $\chi_{N,0} < \chi_{N,1} < \ldots$
such that for each nonnegative integer $n$, the differential equation
  \begin{equation}
L_{N,c}[\varphi](x) + \chi_{N,n}\varphi(x) = 0
    \label{7.70}
  \end{equation}
has a solution which is bounded and twice continuously differentiable on the
interval $(0,1)$, so that $\varphi'(0)=0$ if $p=-1$ and $N=0$, and
$\varphi(0)=0$ otherwise. 
\end{theorem}
Generalized prolate spheroidal functions (GPSFs) were defined in \cite{bell4} 
as functions $\varphi$ satisfying equation \eqref{7.70}.

The following theorem is proved in~\cite{bell4}.
\begin{theorem} \label{thm7.2}
Suppose that $c > 0$, $N$ is a nonnegative integer, and the
operators $M_{N,c}$ and $L_{N,c}$ are defined via~{\rm(\ref{7.51})}
and~{\rm(\ref{7.60})} respectively. Suppose also that $\varphi\colon (0,1)\to
\R$ is in $L^2([0,1])$, is twice differentiable, and that $\varphi'(0)=0$ if
$p=-1$ and $N=0$, and $\varphi(0)=0$ otherwise.
Then
  \begin{equation}
L_{N,c}\bigl[ M_{N,c}[\varphi] \bigr](x) = 
M_{N,c}\bigl[ L_{N,c}[\varphi] \bigr](x),
    \label{7.80}
  \end{equation}
for all $0< x< 1$.
\end{theorem}
\begin{remark}\label{660}
Since Theorem~\ref{thm7.1} shows that the eigenvalues of $L_{N,c}$ are not
degenerate, Theorem~\ref{thm7.2} implies that $L_{N,c}$ and $M_{N,c}$ have
the same eigenfunctions.
\end{remark}

\subsubsection{Zernike polynomials and GPSFs} \label{sect8}

In this section we describe the relationship between Zernike polynomials and
GPSFs. We use $\varphi_{N,n}^c$, where $c > 0$ and $N$ and $n$ are
arbitrary nonnegative integers, to denote the $n$th eigenfunction of
$L_{N,c}$, defined in~(\ref{7.60}); we denote by $\chi_{N,n}(c)$ the
eigenvalue corresponding to eigenfunction $\varphi_{N,n}^c$.

For $c=0$, the eigenfunctions and eigenvalues of the differential operator
$L_{N,c}$, defined in~(\ref{7.60}), are 
\begin{equation}\label{8.10}
\tbar_{N,n}(r)
\end{equation}
and
\begin{equation}\label{8.20}
\chi_{N,n}(0) = (N+\tfrac{p}{2}+2n+\tfrac{1}{2})
(N+\tfrac{p}{2}+2n+\tfrac{3}{2}),
\end{equation}
respectively (see Lemma \ref{lem5.20}), where $\tbar_{N,n}$ is defined in (\ref{400}),
 $0\le r\le 1$, and $N,n$ are arbitrary nonnegative
integers. 

For small $c>0$, the connection between Zernike polynomials 
and GPSFs is given by the formulas
  \begin{equation}
\varphi_{N,n}^c(r) = \overline{T}_{N,n}(r) + o(c^2),
    \label{8.40}
  \end{equation}
and
  \begin{equation}
\chi_{N,n}(c) = \chi_{N,n}(0) + o(c^2),
    \label{8.50}
  \end{equation}
as $c\to 0$, where $0\le r\le 1$ and $N$ and $n$ are arbitrary nonnegative
integers (see~\cite{bell4}).

For $c>0$, the functions $T_{N,n}$ are also related to the integral operator
$M_{N,c}$, defined in~(\ref{7.51}), via the formula
  \begin{align}
&\hspace*{-3em} M_{N,c}\bigl[ T_{N,n} \bigr](x)
= \int_0^1 J_{N+p/2}(cxy) \sqrt{cxy}\, T_{N,n}(y)\, dy
= \frac{(-1)^n J_{N+p/2+2n+1}(cx)}{\sqrt{cx}},
    \label{8.60}
  \end{align}
where $x\ge 0$ and $N$ and $n$ are arbitrary nonnegative integers 
(see Equation (85) in \cite{pkzern}).

\section{Analytical apparatus}\label{secanap}
In this section, we provide analytical apparatus relating to GPSFs
that will be used in numerical schemes in subsequent sections. 

\subsection{Properties of GPSFs}

The following theorem provides a formula for the ratios of eigenvalues
$\beta_{N,n}$ (see~(\ref{5.100})), and is used in the numerical
evaluation of $\beta_{N,n}$. A proof follows immediately from 
Theorem 7.1 of \cite{osipov2}.
\begin{theorem}\label{2900}
Suppose that $N$ is a nonnegative integer. Then
\begin{equation}\label{3140}
\frac{\beta_{N,m}}{\beta_{N,n}} =
\frac{\int_0^1 x \Phi_{N,n}'(x) \Phi_{N,m}(x) x^{p+1}\, dx}
{\int_0^1 x \Phi_{N,m}'(x) \Phi_{N,n}(x) x^{p+1}\, dx},
\end{equation}
for each nonnegative integers $n$ and $m$.
\end{theorem}
\subsection{Decay of the expansion coefficients of GPSFs in Zernike
polynomials}\label{secdecay}

Since the functions $\Phi_{N,n}$ are analytic on $\C$ for all
nonnegative integers $N$ and $n$ (see Observation~\ref{obs5.2}), and
$\Phi_{N,n}^{(k)}(0)=0$ for $k=0,1,\ldots,N-1$ (see Theorem~\ref{thm12.55}),
 the functions $\Phi_{N,n}$ are representable by a series of Zernike
polynomials of the form
\begin{align}\label{940}
\Phi_{N,n}(r)=\sum_{k=0}^\infty a_{n,k} \overline{R}_{N,k}(r),
\end{align}
for all $0\le r\le 1$, where $a_{n,0},a_{n,1},\ldots$ satisfy
\begin{equation}\label{900}
a_{n,k}=\int_0^1 \rbar_{N,k}(r) \Phi_{N,n}(r) dr
\end{equation}
where $\overline{R}_{N,n}$ is defined in~(\ref{6.60}).
The following technical lemma will be used in the proof of 
Theorem \ref{1860}.
\begin{lemma}\label{1920}
For any integer $p\geq -1$, for all $c>0$, and for all $\rho \in [0,1]$,
\begin{equation}
\left| \int_0^1 \frac{J_{N+p/2}(cr\rho)}{(cr\rho)^{p/2}}
\rbar_{N,k}(r)r^{p+1}dr \right|  
< \left( \frac{1}{2} \right)^{N+p/2+2k+1}
\end{equation}
for any non-negative integers $N,k$ such that $N+2k\geq ec$
where $\rbar_{N,n}$ is defined in (\ref{6.60}) and $J_{N+p/2}$ is 
a Bessel function of the first kind. 
\end{lemma}
\begin{proof}
According to equation (85) in \cite{pkzern},
\begin{equation}\label{1720}
\int_0^1 \frac{J_{N+p/2}(cr\rho)}{(cr\rho)^{p/2}} R_{N,k}(r) r^{p+1}dr
= \frac{(-1)^n J_{N+p/2+2k+1}(c\rho)}{(c\rho)^{p/2+1}},
\end{equation}
where $J_{N+p/2}$ is a Bessel function of the first kind. 
Applying Lemma \ref{1840} to (\ref{1720}), we obtain
\begin{equation}\label{1820}
\left| \int_0^1 \frac{J_{N+p/2}(cr\rho)}{(cr\rho)^{p/2}}
\rbar_{N,k}(r)r^{p+1}dr \right| \leq 
\frac{(c\rho/2)^{N+p/2+2k+1}}{(c\rho)^{p/2+1}}
\frac{\sqrt{2(N+p/2+2k+1)}}{\Gamma(N+p/2+2k+2)}.
\end{equation}
Combining Lemma \ref{1800} and (\ref{1820}), we have
\begin{equation}
\begin{split}
\left| \int_0^1 \frac{J_{N+p/2}(cr\rho)}{(cr\rho)^{p/2}}
\rbar_{N,k}(r)r^{p+1}dr  \right|
&\leq \left(\frac{1}{2}\right)^{N+p/2+2k+1} (c\rho)^{N+2k}
\frac{\sqrt{2(2k+N)}}{\Gamma(2k+N+1)}\\
&\leq \left( \frac{1}{2}\right)^{N+p/2+2k+1}
\end{split}
\end{equation}
for $N+2k \geq ec$. 
\end{proof}
The following theorem shows that the coefficients $a_{N,k}$ 
of GPSFs in a Zernike polynomial basis decay 
exponentially and establishes a bound for the decay rate. 
\begin{theorem}\label{1860}
For all non-negative integers $N,n,k$ and for all $c>0$,
\begin{equation}\label{1900}
\bigg| \int_0^1 \Phi_{N,n}(r) \rbar_{N,k} r^{p+1}dr \bigg|
< (p+2)^{-1/2}(\beta_{N,n})^{-1} 
\left( \frac{1}{2} \right)^{N+p/2+2k+1} 
\end{equation}
where $N+2k\geq ec$. 
\end{theorem}
\begin{proof}
Combining (\ref{7.20}) and (\ref{900}), we have
\begin{equation}\label{1880}
\begin{split}
\bigg| \int_0^1 \Phi_{N,n}(r) \rbar_{N,k} r^{p+1} dr \bigg| &\\
= \bigg| \int_0^1 (\beta_{N,n})^{-1} &\left(\int_0^1 \frac{J_{N+p/2}(cr\rho)}
{(cr\rho)^{p/2}}\Phi_{N,n}(\rho)\rho^{p+1}d\rho \right)
\rbar_{N,k}(r)r^{p+1}dr \bigg|.
\end{split}
\end{equation}
Changing the order of integration of (\ref{1880}),
\begin{equation}\label{1940}
\begin{split}
\bigg| \int_0^1 \Phi_{N,n}(r) \rbar_{N,k} r^{p+1} dr \bigg| &\\
= (\beta_{N,n})^{-1}  &\int_0^1 \bigg|  \Phi_{N,n}(\rho)\rho^{p+1} \bigg| 
\bigg| \int_0^1 \frac{J_{N+p/2}(cr\rho)}
{(cr\rho)^{p/2}}
\rbar_{N,k}(r)r^{p+1}dr \bigg| d\rho  .
\end{split}
\end{equation}
Applying Lemma \ref{1920} to (\ref{1940}) and applying Cauchy-Schwarz,
we obtain
\begin{equation}
\begin{split}
\bigg| \int_0^1 \Phi_{N,n}(r) \rbar_{N,k} r^{p+1} dr \bigg|
&\leq (\beta_{N,n})^{-1} \left( \frac{1}{2} \right)^{N+p/2+2k+1} 
\int_0^1 \Phi_{N,n}(r)r^{p+1}dr \\
&\leq (2p+3)^{-1/2}(\beta_{N,n})^{-1} 
\left( \frac{1}{2} \right)^{N+p/2+2k+1}.
\end{split}
\end{equation}
for $N+2k \geq ec$. 
\end{proof}
In the following remark we summarize the proof of Theorem \ref{1860}.
\begin{remark}
The preceding proof bounds $|\inner{\Phi_{N, n}}{\rbar_{N, n}}_{r^{p+1}} |$, 
by first observing that 
\begin{align}
\Phi_{N, n} = \beta_{N, n}^{-1} H_{N, c}[\Phi_{N, n}]
\end{align}
where $\inner{\cdot}{\cdot}_{r^{p+1}}$ denotes the $L^2$ inner product with 
weight function $r^{p+1}$.
We then use that $H_{N, c}$ is self-adjoint to obtain 
\begin{align}
| \inner{\Phi_{N, n}}{\rbar_{N, n}}_{r^{p+1}} |
= \beta_{N, n}^{-1} | \inner{H_{N, c} [\rbar_{N, n}]}{\Phi_{N, n}}_{r^{p+1}} |.
\end{align}
Next, we substitute an $L^{\infty}$ bound on $H_{N, c} [\rbar_{N, n}]$ (Lemma \ref{1920}) to obtain
\begin{align}\label{proof_sum1}
\inner{\Phi_{N, n}}{\rbar_{N, n}} \leq \beta_{N, n}^{-1} \| H_{N, c} [\rbar_{N, n}] \|_{\infty} 
\inner{| \Phi_{N, n} | }{r^{p+1}}
\end{align}
where the inner product in \eqref{proof_sum1} is the usual $L^2$ inner product. 
The result follows immediately from a bound on $ \| H_{N, c} [\rbar_{N, n}] \|_{\infty} $ and 
applying Cauchy-Schwarz to $\inner{| \Phi_{N, n} | }{r^{p+1}}$.
\end{remark}
\subsection{Tridiagonal nature of $L_{N,c}$}\label{sectridiag}
In this section, we show that in the basis of $\tbar_{N,n}$ 
(see (\ref{400})), the matrix representing differential operator 
$L_{N,c}$ (see (\ref{7.60})) is symmetric and tridiagonal.
A similar set of observations is made in Section 4 of \cite{xiao} for the purpose of 
computing with prolate spheroidal wave functions defined on the 
interval. This section generalizes several of those observations to arbitrary 
dimensions.

The following lemma provides an identity relating
the differential operator $L_{N,c}$ to $\tbar_{N,n}$. 
\begin{lemma}\label{880}
For all non-negative integers $N,n$ and real numbers $c>0$
\begin{equation}\label{220}
L_{N,c}[\tbar_{N,n}]=-\chi_{N,n}\tbar_{N,n}(x)-c^2x^2\tbar_{N,n}(x)
\end{equation}
for all $x\in [0,1]$ where $\chi_{N,n}$ is defined in (\ref{210}) 
and $L_{N,c}$ is defined in (\ref{7.60}). 
\end{lemma}
\begin{proof}
Applying $L_{N,c}$ to $\tbar_{N,n}$, we obtain
\begin{equation}\label{420}
L_{N,c}[\tbar_{N,n}](x)=(1-x^2) \tbar_{N,n}^{\prime\prime}(x)-
2x\tbar^{\prime}_{N,n}(x)+ 
\biggl( \frac{\frac{1}{4} - (N+\frac{p}{2})^2}{x^2}
  - c^2 x^2 \biggr) \tbar_{N,n}(x).
\end{equation}
Identity (\ref{220}) follows immediately from the combination
of (\ref{200}) and (\ref{420}).
\end{proof}
The following theorem follows readily from the combination of 
Lemma \ref{880} and Lemma \ref{860}.
\begin{theorem}\label{480}
For any non-negative integer $N$, any integer $n\geq 1$,
and for all $r\in (0,1)$,
\begin{equation}\label{280}
\begin{aligned}
L_{N,c}[\tbar_{N,n}]=a_n\tbar_{N,n-1}(r)+b_n\tbar_{N,n}(r)
+c_n\tbar_{N,n+1}(r)
\end{aligned}
\end{equation}
where 
\begin{equation}\label{290}
\begin{aligned}
a_n=&\frac{-c^2(n+N+p/2)n}
{(2n+N+p/2)\sqrt{2n+N+p/2+1}\sqrt{2n+N+p/2-1}}\\
b_n=&\frac{-c^2(N+p/2)^2}{2(2n+N+p/2)(2n+N+p/2+2)}
-\frac{c^2}{2}+\chi_{N,n}\\
c_n=&\frac{-c^2(n+1+N+p/2)(n+1)}
{(2n+N+p/2+2)\sqrt{2n+N+p/2+3}\sqrt{2n+N+p/2+1}}
\end{aligned}
\end{equation}
and $\chi_{N,n}$ is defined in (\ref{210}).
\end{theorem}
\begin{observation}\label{obs:tridiag}
It follows immediately from Theorem \ref{480} that the matrix 
corresponding to the differential operator $L_{N,c}$ acting
on the $\tbar_{N,n}$ basis is symmetric and tridiagonal.
Specifically, for any positive integer $n$ and for all $r \in (0,1)$,
\begin{equation}
\begin{bmatrix}\label{640}
    b_{0} & c_{0}   &        &        &          & 0      \\
    c_{0} & b_{1}   & c_{1}  &        &          &        \\
          & c_{1}   & b_{2}  & c_{2}  &          &        \\
          &         & \ddots & \ddots & \ddots   &        \\
          &         &        & c_{n-2}& b_{n-1}  & c_{n-1}\\
    0     &         &        &        & c_{n-1}  & b_{n}
\end{bmatrix}
\begin{bmatrix}
    \tbar_{N,0}(r)     
           \\
           \\
    \vdots \\
           \\
           \\
    \tbar_{N,n}(r)     
\end{bmatrix}
+
\begin{bmatrix}
    0      \\
           \\
    \vdots \\
           \\
    0      \\
    c_{n}\tbar_{N,n+1}(r)
\end{bmatrix}
=
\begin{bmatrix}
    \tbar_{N,0}(r)     
           \\
           \\
    \vdots \\
           \\
           \\
    \tbar_{N,n}(r)     
\end{bmatrix}
\end{equation}
where $b_k$ and $c_k$ are defined in (\ref{290}) and $\tbar_{N,k}$ is
defined in (\ref{400}).
\end{observation}
\begin{observation}
Let $A$ be the infinite symmetric tridiagonal matrix satisfying
$A_{1,1}=b_0$, $A_{1,2}=c_0$ and for all integers $k\geq 2$,
\begin{equation}\label{1980}
\begin{split}
A_{k,k-1}&=c_{k-1}\\
A_{k,k}&=b_k\\
A_{k,k+1}&=c_k,
\end{split}
\end{equation}
where $b_k,c_k$ are defined in (\ref{290}). That is,
\begin{equation}\label{1960}
A=
\begin{bmatrix}
    b_{0} & c_{0}   &        &        &           \\
    c_{0} & b_{1}   & c_{1}  &        &           \\
          & c_{1}   & b_{2}  & c_{2}  &           \\
          &         & \ddots & \ddots & \ddots    \\
\end{bmatrix}
.
\end{equation}
Suppose further that we define the infinite vector $a_n$ by 
the equation
\begin{equation}
a_n=(a_{n,0},a_{n,1},...)^T,
\end{equation}
where $a_{n,k}$ is defined in (\ref{900}). 
By the combination of Theorem \ref{thm7.2} and Remark \ref{660}, 
we know that $\varphi_{N,n}$ is the 
eigenfunction corresponding to $\chi_{N,n}(c)$, the $n$th 
smallest eigenvalue of differential operator $L_{N,c}$.
Therefore,
\begin{equation}\label{2600}
A a_n =\chi_{N,n}(c) a_n.
\end{equation}
Furthermore, the $a_{n,k}$ decay exponentially in $k$
(see Theorem \ref{1860}).
\end{observation}
\begin{remark}\label{rem:gpsf_eval}
The eigenvalues $\chi_{N,n}$ of differential operator $L_{N,c}$
and the coefficients in the Zernike expansion of the eigenfunctions
$\Phi_{N,n}$ can be computed numerically to high relative precision
by the following process. First, we reduce the infinite
dimensional matrix $A$ (see (\ref{1960})) to $A_K$, its upperleft 
$K \times K$ 
submatrix where $K$ is chosen, using Theorem \ref{1860}, so that
$a_{n,K-1}$ is smaller than machine precision and is in the regime 
of exponential decay. We can then use standard algorithms to
find the eigenvalues and eigenvectors of matrix $A_K$. 
See Algorithm \ref{2040} for a more detailed description of the 
algorithm. 
We use this approach because discretizing integral operator 
$H_{N, c}$ (see \eqref{5.90}) would lead to nearly degenerate 
eigenvalues and thus the unstable evaluation of eigenfunctions.
\end{remark}
\section{Numerical evaluation of GPSFs}\label{secnumev}
In this section, we describe an algorithm (Algorithm \ref{2040})
for the evaluation of $\Phi_{N,n}(r)$ (see (\ref{5.100})) for 
all $r \in [0,1]$.
\begin{algorithm1}[Evaluation of $\Phi_{N, n}$]\,\label{2040}
\begin{enumerate}
\item
Use Theorem \ref{1860} to determine $K$ of Remark \ref{rem:gpsf_eval}, 
the number of terms needed in a Zernike expansion of $\Phi_{N,n}$ to 
achieve some desired precision. 

\item 
Generate $A_K$, the symmetric, tri-diagonal, upper-left 
$K \times K$ sub-matrix of $A$ (see (\ref{1960})). \\

\item
Use an eigensolver to find the eigenvector, $\tilde{a}_{n}$,
corresponding to the $n+1^{\text{th}}$ largest eigenvalue, 
$\tilde{\chi}_{N,n}$. That is, find $\tilde{a}_n$ and $\tilde \chi_{N,n}$
such that
\begin{equation}\label{2880}
A_K\tilde{a}_n=\tilde{\chi}_{N,n}\tilde{a}_n
\end{equation}
where we define the components of $\tilde{a}_{N,n}$ by the formula,
\begin{equation}\label{2020}
\tilde{a}_{n}=(a_{n,0},a_{n,1},...,a_{n,K-1}).
\end{equation}

\item
Evaluate $\Phi_{N,n}(r)$ by the expansion 
\begin{equation}
\Phi_{N,n}(r)=\sum_{i=0}^{k}a_{n,i} \rbar_{N,i}(r)
\end{equation}
where, $\rbar_{N,i}$ is evaluated via Lemma \ref{lem755}
and $a_{n,i}$ are the components of eigenvector (\ref{2020})
recovered in Step 3. 
\end{enumerate}
\end{algorithm1}
\begin{remark}\label{rem:rel_ev}
It turns out that because of the structure of $A_K$, standard numerical
algorithms (for example, the inverse power method) compute the components 
of eigenvector $\tilde{a}_n$ (see \eqref{2020}), and thus the coefficients
of a GPSF in a Zernike expansion, 
to high relative precision. In particular, the components of $\tilde a_n$
that are of magnitude far less than machine precision are computed to 
high relative precision. For example, when using double-precision 
arithmetic, a component of $\tilde a_n$ of magnitude $10^{-100}$ 
will be computed in absolute precision to $116$ digits. 
This fact is proved in a more general setting in \cite{osipov}. 
\end{remark}
\subsection{Numerical evaluation of the single eigenvalue $\beta_{N,i}$}
\label{sec:one_eig}
In this section, we describe a numerical method to evaluate
the eigenvalue $\beta_{N,n}$ (see \eqref{5.100}) for 
fixed $n$ to high relative precision. 

The following is a technical lemma will be used in the proof 
of Theorem \ref{960}.
\begin{lemma}\label{780}
For all non-negative integers $N,k$,
\begin{equation}\label{720}
\int_0^1 r^{N}\varphi_{N,k}(r)r^{\frac{p+1}{2}}dr=
\frac{a_{k,0}}{\sqrt{2N+p+2}}
\end{equation}
where $\varphi_{N,k}$ is defined in (\ref{7.30}) and 
$p\geq-1$ is an integer. 
\end{lemma}
\begin{proof}
Using (\ref{6.20}),
\begin{equation}\label{700}
\int_0^1 r^{N}\varphi_{N,k}(r)r^{\frac{p+1}{2}}dr=
\int_0^1 R_{N,0}(r)\varphi_{N,k}(r)r^{\frac{p+1}{2}}dr
\end{equation}
Applying (\ref{900}) and (\ref{400}) to (\ref{700}), we obtain
\begin{equation}
\int_0^1 r^{N}\varphi_{N,k}(r)r^{\frac{p+1}{2}}dr=
\frac{1}{\sqrt{2N+p+2}} 
\int_0^1 \tbar_{N,0}(r)\varphi_{N,k}(r)dr
=\frac{a_{k,0}}{\sqrt{2N+p+2}}.
\end{equation}
\end{proof}
We will denote by $\tilde{\varphi}_{N,n}(r)$ the function on $[0,1]$ 
defined by the formula
\begin{equation}\label{80}
\tilde{\varphi}_{N,n}(r) = \frac{\varphi_{N,n}(r)}{r^{N+\frac{p+1}{2}}}
\end{equation}
where $N,n$ are non-negative integers. We introduce 
$\tilde{\varphi}_{N, n}$ to remove the 
leading power of $\varphi_{N, n}$ at the origin to make its value non-zero.

The following identity will be used in the proof of Theorem \ref{960}.
\begin{lemma}
For all non-negative integers $N,k$,
\begin{equation}\label{90}
\tilde{\varphi}_{N,k}(0) = \sum_{i=0}^\infty a_{k,i} \sqrt{2(2i+N+p/2+1)}
(-1)^i\dbinom{i+N+p/2}{i}.
\end{equation}
where $\tilde{\varphi}_{N,k}$ is defined in (\ref{80}) and $a_{k,i}$
is defined in (\ref{900}).  
\end{lemma}
\begin{proof}
Combining (\ref{80}) and (\ref{600}), we have
\begin{equation}\label{800}
\begin{split}
\tilde{\varphi}_{N,k}(r)&=\frac{\varphi_{N,k}(r)}{r^{N+\frac{p+1}{2}}}
=\sum_{i=0}^\infty a_{k,i} \frac{\tbar_{N,i}(r)}{r^{N+\frac{p+1}{2}}}
=\sum_{i=0}^\infty a_{k,i} \tilde{T}_{N,i}(r)
\end{split}
\end{equation}
where $\tilde{T}_{N, n}$ is defined in (\ref{600}) and $\tbar_{N,n}$
is defined in (\ref{400}). Identity (\ref{90}) follows immediately
from applying Lemma \ref{820} to (\ref{800}) and setting $r=0$.
\end{proof}
\begin{lemma}
For all non-negative integers $N$ and $p \geq -1$, 
\begin{align}\label{mphit_0}
M_{N, c}[\tilde{\varphi}_{N, k}](0) = a_{k,0}c^{N+ \frac{p+1}{2}}(2^{N+p/2}\Gamma(N+p/2+1)\sqrt{2N+p+2})^{-1}
\end{align}
where $M_{N, c}$ is defined in \eqref{7.51}, $\tilde{\varphi}_{N, k}$ is defined in \eqref{80}, and $a_{k,0}$ is defined in \eqref{900}.
\end{lemma}
\begin{proof}
The Bessel function of the first kind $J_{N+p/2}$ is given by 
\begin{equation}\label{110}
J_{N+p/2} (cr\rho) =\left(\frac{cr\rho}{2}\right)^{N+p/2}
\sum_{k=0}^\infty \frac{(-(cr\rho)^2/4)^{k}}{k!\Gamma(N+p/2+k+1)}
\end{equation}
where $\Gamma(n)$ is the gamma function (see formula 10.2.2 of \cite{dlmf}).
Dividing both sides of (\ref{7.51}) by 
$r^{N+\frac{(p+1)}{2}}$, we obtain the equation
\begin{equation}\label{mphi_tilde}
M_{N, c}[\tilde{\varphi}_{N, k}] 
= \int_0^1 \frac{J_{N+p/2}(cr\rho)}{r^{N+\frac{p}{2}}} \sqrt{c\rho} 
\varphi_{N,k}(\rho)d\rho
\end{equation}
where $\tilde{\varphi}_{N,k}$ is defined in (\ref{80}). Setting $r=0$ 
in \eqref{mphi_tilde} and subsituting in \eqref{110},
we obtain 
\begin{align}\label{mphit_0_last}
M_{N, c}[\tilde{\varphi}_{N, k}](0) =
\int_0^1 \left(\frac{c\rho}{2}\right)&^{N+p/2}\frac{(c\rho)^{1/2}}
{\Gamma(N+p/2+1)}\varphi_{N,k}(\rho)d\rho
\end{align}
Equation \eqref{mphit_0} follows immediately from applying 
Lemma \ref{780} and \eqref{7.40} to \eqref{mphit_0_last}.
\end{proof}

The following theorem provides a formula that can be used to compute  
$\beta_{N,n}$  (see (\ref{7.40})), an eigenvalue of 
integral operator $H_{N,c}$ (see \eqref{5.90}). It follows immediately
from combining \eqref{7.40}, \eqref{7.50}, \eqref{7.51}, \eqref{mphit_0}, and \eqref{90}. 
\begin{theorem}\label{960}
For all $c > 0$, $p \geq -1$, and non-negative integers $N,k$,
\begin{equation}\label{150}
\begin{aligned}
\beta_{N,k}= c^{-\frac{p+1}{2}} \frac{M[\tilde{\varphi}_{N, k}](0)}{\tilde{\varphi}_{N, k}(0)}
= 
\frac{
a_{k,0}c^N(2^{N+p/2} \, \Gamma(N+p/2+1)\, \sqrt{2N+p+2}\, )^{-1}
}
{
\sum\limits_{i=0}^\infty a_{k,i} \sqrt{2(2i+N+p/2+1)}
(-1)^i\dbinom{i+N+p/2}{i}
}
\end{aligned}
\end{equation}
where $M_{N, c}$, is defined in \eqref{7.51}, $\beta_{N,k}$ is defined in (\ref{7.20}) and 
$a_{k,i}$ is defined in (\ref{900}). 
\end{theorem}
%
%
\begin{remark}\label{3180}
For any non-negative integers $N,k$, the eigenvalue $\beta_{N,k}$ 
can be evaluated stably by first using Algorithm \ref{2040} 
to compute the eigenvector $\tilde{a}_{k}$ (see \eqref{2020}),
and then evaluating $\beta_{N,k}$ via sum (\ref{150})
where $\tilde{a}_{k}$ are approximations to $a_k$. In (\ref{150}),
the sum 
\begin{equation}\label{3420}
\sum_{i=0}^\infty a_{k,i} \sqrt{2(2i+N+p/2+1)}
(-1)^i\dbinom{i+N+p/2}{i}
\end{equation}
can be computed to high relative precision by truncating the sum
at a point when the next term is no more than
machine precision times the sum up to that point. 
The eigenvalue $\beta_{N, k}$ can be evaluated to high relative precision even 
when it is of magnitude far less than machine precision if the components of 
eigenvector $a_{k}$ are evaluated to high relative precision (see Remark \ref{rem:rel_ev}).
\end{remark}
%

\begin{remark}
Theorem \ref{960} is a generalization to arbitrary dimension 
of equations (10.15) and (10.16) of \cite{osipov2}, which allow for
efficiently evaluating a single eigenvalue corresponding to the one-dimensional 
prolate spheroidal wave functions. 
%
\end{remark}

\subsection{Numerical evaluation of eigenvalues 
$\beta_{N,0},\beta_{N,1},...,\beta_{N,k}$}\label{sec:all_eigs}
In this section, we describe an algorithm (Algorithm \ref{3260})
for numerically evaluating the eigenvalues 
$\beta_{N,0},\beta_{N,1},...,\beta_{N,k}$ (see (\ref{5.100}))
for any non-negative integers $N,k$.
It is possible to use the tools of Section \ref{sec:one_eig} to evaluate 
each of the eigenvalues $\beta_{N,0},\beta_{N,1},...,\beta_{N,k}$ one 
at a time. 
However, the algorithm we introduce in this section, Algorithm \ref{3260},
is more efficient for this task.

%
\begin{observation}\label{2980}
Here we show that for all non-negative integers $N, K$ and $x_0,...,x_K \in \R$, 
equation \eqref{2940} can be used to stably 
and efficiently construct $\alpha_0,...,\alpha_K$ such that 
\begin{equation}
\sum_{i=0}^K \alpha_i \tbar_{N,i}(r)
\end{equation}
is an accurate approximation of 
\begin{equation}\label{3660.2}
\sum_{i=0}^K x_i r\tbar_{N,i}^\prime(r)
\end{equation}
for all $0 \leq r \leq 1$. 
The approximation can be constructed as follows. 
Fix $\epsilon>0$ and let $x_0,...,x_K$ be a sequence of real numbers such that 
\begin{equation}\label{k1}
\sum_{i=K_1+1}^K |x_k| < \epsilon
\end{equation}
where $0 \leq K_1 \leq K$. Using (\ref{400}) and (\ref{6.60}), we have
\begin{equation}\label{3660}
\sum_{i=0}^K x_i \tbar_{N,n}(r)=\sum_{i=0}^K \alpha_iT_{N,n}(r)
\end{equation}
where $x_0,...,x_K$ are real numbers and $\alpha_i$ is defined by 
the formula
\begin{equation}
\alpha_i=x_i\sqrt{2(2i+N+p/2+1)}
\end{equation}
for $i=0,1,...,K$. 
%
Scaling both sides of (\ref{2940}) by $\alpha_{0} / \tilde{a}_{1}$ and 
setting $n = 1$, we obtain
\begin{equation}\label{3040}
\begin{split}
&\alpha_0rT_{N,0}^\prime(r)-\frac{\alpha_0 \tilde b_1}{\tilde a_1}
rT_{N,1}^\prime(r)
+ \frac{\alpha_0 \tilde c_1}{\tilde a_1}rT_{N,2}^\prime(r) \\
&= \frac{\alpha_0 a_1}{\tilde a_1} T_{N,0}(r) 
- \frac{\alpha_0 b_1}{\tilde a_1}T_{N,1}(r)
+ \frac{\alpha_0 c_1}{\tilde a_1}T_{N,2}(r)
\end{split}
\end{equation}
where $a_i,b_i,c_i,\tilde a_i,\tilde b_i,\tilde c_i$ are defined in 
Lemma \ref{2920}.
Scaling (\ref{2940}) with setting $n = 2$ and adding the resulting equation 
to (\ref{3040}), we obtain
\begin{equation}\label{3640}
\begin{split}
&\alpha_0rT_{N,0}^\prime(r)+\alpha_1rT_{N,1}^\prime(r)
+ \left( \frac{\alpha_0 \tilde c_1}{\tilde a_1}
-\frac{\tilde b_2}{\tilde a_2}  \left( \frac{\alpha_0 \tilde b_1}{\tilde a_1} 
+ \alpha_1 \right) \right)rT_{N,2}^\prime(r)\\
&+ \left( \left( \frac{\alpha_0 \tilde b_1}{\tilde a_1} + \alpha_1 \right)
\tilde a_2^{-1} \right)
\left(\tilde c_2rT_{N,3}^\prime(r)\right).\\
&= \frac{\alpha_0 a_1}{\tilde a_1} T_{N,0}(r) 
- \frac{\alpha_0 b_1}{\tilde a_1}T_{N,1}(r)
+ \frac{\alpha_0 c_1}{\tilde a_1}T_{N,2}(r)\\
&+ \left( \left( \frac{\alpha_0 \tilde b_1}{\tilde a_1} + \alpha_1 \right)
\tilde a_2^{-1} \right)
\left(a_2 T_{N,1}(r)- b_2T_{N,2}(r)
+ c_2T_{N,3}(r)\right).
\end{split}
\end{equation}
We note that the coefficients of the first two terms on the left-hand-side 
of (\ref{3640}) coincide with the coefficients of the first two terms of (\ref{3660.2}).

We continue by adding scaled versions of (\ref{2940}) 
to (\ref{3640}) until the expansion on the left hand 
side of (\ref{3640}) approximates (\ref{3660}). 
After $K_1+1$ steps, the resulting expansion will be accurate to approximately
$\epsilon$ precision. 
Specifically, at the start of step $k$, for $2\leq k \leq K_1+1$, we
have
\begin{equation}\label{2960}
\begin{split}
&\alpha_0rT_{N,0}^\prime(r)+\alpha_1rT_{N,1}^\prime(r)+...
+ \alpha_{k-1}rT_{N,k-1}^\prime(r)+ c_{k}rT_{N,k}^\prime(r)
+ c_{k+1}rT_{N,k+1}^\prime(r)\\
&=
y_0T_{N,0}(r)+y_1T_{N,1}(r)+...+y_{k}T_{N,k}(r)+y_{k+1}T_{N,k}(r)
\end{split}
\end{equation}
for some real numbers $c_{k},c_{k+1},y_0,y_1,...,y_{k+1}$.
Scaling both sides of (\ref{2940}) and adding the resulting equation 
to (\ref{2960}), we obtain
\begin{equation}\label{3100}
\begin{split}
&\alpha_0rT_{N,0}^\prime(r)+\alpha_1rT_{N,1}^\prime(r)+...
+ \alpha_{k-2}rT_{N,k-2}^\prime(r) 
+ \alpha_{k-1}rT_{N,k-1}^\prime(r) \\
&+\left(\frac{-x_{k-1}+\alpha_{k-1}}{\tilde a_{k}}(-\tilde b_k)+x_k\right)
rT_{N,k}^\prime(r)
+\left(\frac{-x_{k-1}+\alpha_{k-1}}{\tilde a_{k}}\tilde c_k \right)
rT_{N,k+1}^\prime(r)\\
&=y_0T_{N,0}+y_1T_{N,1}+...
 + \left(\frac{-x_{k-1}+\alpha_{k-1}}{\tilde a_{k}}a_k +y_{k-1} \right) 
T_{N,k-1}(r)\\
& + \left(\frac{-x_{k-1}+\alpha_{k-1}}{\tilde a_{k}}(-b_k) +y_{k} \right) 
T_{N,k}(r)
 + \left(\frac{-x_{k-1}+\alpha_{k-1}}{\tilde a_{k}}c_k \right) 
T_{N,k+1}(r).
\end{split}
\end{equation}
\end{observation}
The following observation, when combined with Observation \ref{2980}, 
provides a numerical scheme for evaluating integrals of the form 
\begin{equation}\label{3240}
\int_0^1 r\Phi_{N,n}^\prime(r) \Phi_{N,m}(r) r^{p+1}dr.
\end{equation}
This scheme will be used in Algorithm \ref{3260}.
\begin{observation}\label{3160}
Suppose that 
\begin{equation}\label{417a}
r\Phi_{N,n}^\prime(r)=\sum_{i=0}^K x_i \rbar_{N,i}(r)
\end{equation}
and 
\begin{equation}\label{417b}
\Phi_{N,m}(r)=\sum_{i=0}^K y_i \rbar_{N,i}(r).
\end{equation}
where $x_i,y_i$ are real numbers. 
Then, substituting \eqref{417a} and \eqref{417b} into \eqref{6.40} and \eqref{6.70}, we have,
\begin{equation}
\int_0^1 r\Phi_{N,n}^\prime(r) \Phi_{N,m}(r) r^{p+1}dr
=\int_0^1 \sum_{i=0}^K x_i \rbar_{N,i}(r)\sum_{i=0}^K y_i 
\rbar_{N,i}(r) r^{p+1} dr
= \sum_{i=0}^K x_i y_i.
\end{equation}
\end{observation}
We now describe an algorithm for evaluating the eigenvalues 
$\beta_{N,0},\beta_{N,1},...,\beta_{N,k}$ for any non-negative 
integers $N,k$.
\begin{algorithm1}[Evaluation of eigenvalues $\beta_{N,0},\beta_{N,1},...,\beta_{N,k}$] \, \label{3260}

\begin{enumerate}
\item 
Use Algorithm \ref{2040} to recover the Zernike expansions 
of the GPSFs
\begin{equation}
\Phi_{N,0},\Phi_{N,1},...,\Phi_{N,n}.\\
\end{equation}

\item
Compute eigenvalue $\beta_{N,0}$ (see (\ref{5.100}))
using Remark \ref{3180}. \\

\item
Use Observation \ref{2980} to evaluate 
the $\rbar_{N,n}$ expansion of $r\Phi_{N,0}^\prime$ and
$r\Phi_{N,1}^\prime$.\\

\item
Use Observation \ref{3160} to compute the integrals
\begin{equation}
\int_0^1 r\Phi_{N,1}^\prime(r)\Phi_{N,0}(r)r^{p+1}dr
\end{equation}
and 
\begin{equation}
\int_0^1 r\Phi_{N,0}^\prime(r)\Phi_{N,1}(r)r^{p+1}dr
\end{equation}
where the Zernike expansions of $\Phi_{N,0}(r), \Phi_{N,1}(r)$ 
were computed in Step 1 and the Zernike expansions of 
$r\Phi_{N,0}^\prime(r), r\Phi_{N,1}^\prime(r)$ were computed in Step 3.\\

\item
Using Theorem \ref{2900}, evaluate $\beta_{N,1}$ using the 
formula 
\begin{equation}\label{3200}
\beta_{N,1}=\beta_{N,0}
\frac{\int_0^1 r \Phi_{N,1}'(r) \Phi_{N,0}(r) r^{p+1}\, dr}
{\int_0^1 r \Phi_{N,0}'(r) \Phi_{N,1}(r) r^{p+1}\, dr}.
\end{equation}
where $\beta_{N,0}$ was obtained in Step 2 and the numerator
and denominator of (\ref{3200}) were evaluated in Step 4. \\

\item
Repeat Steps 3-5 $k$ times, each time computing the next eigenvalue,
$\beta_{N,i+1}$ via the formula
\begin{equation}\label{3280}
\beta_{N,i+1}=\beta_{N,i}
\frac{\int_0^1 r \Phi_{N,i+1}'(r) \Phi_{N,i}(r) r^{p+1}\, dr}
{\int_0^1 r \Phi_{N,i}'(r) \Phi_{N,i+1}(r) r^{p+1}\, dr}.
\end{equation}
\end{enumerate}
\end{algorithm1}
\begin{remark}\label{rem:alg_more_eff}
The eigenvalues $\beta_{N, 0},...,\beta_{N, k}$ could be computed to high 
relative precision one-by-one via repeated application of equation \eqref{150}.
This procedure would require the evaluation of the eigenvectors 
$\tilde{a}_{1},...,\tilde{a}_{k+1}$ (see \eqref{2020}) to high relative precision
(see Remark \ref{rem:rel_ev}). Numerical and analytical tools for this task are 
provided in \cite{osipov}.
On the other hand, Algorithm \ref{3260} requires only absolute precision 
in $\tilde{a}_{1},...,\tilde{a}_{k+1}$. 
\end{remark}

\section{Quadratures for band-limited functions}\label{secquads}
In this section, we describe a quadrature scheme for 
bandlimited functions using nodes that are a tensor product of
roots of GPSFs in the radial direction and nodes that integrate
spherical harmonics in the angular direction. 

The following lemma shows that a quadrature rule that accurately
integrates complex exponentials, also integrates bandlimited 
functions accurately.
\begin{lemma}\label{1520}
Let $\xi_1,...,\xi_n \in B$ and $w_1,...,w_n \in \R$ satisfy
\begin{equation}\label{1300}
\left| \int_B e^{ic \inner{x}{t}} dt -
\sum_{i=1}^n w_i e^{ic \inner{x}{\xi_i}} \right| < \epsilon
\end{equation}
for all $x\in B$ where $B$ denotes the unit ball in $\R^n$ for any 
non-negative integer $n$ and $\epsilon>0$ is fixed. 
Then, for all $f: B \rightarrow \C$ such that 
\begin{equation}\label{1300.1}
f(x)=\int_B \sigma(t) e^{ic \inner{x}{t}} dt
\end{equation}
where $\sigma \in L^2(B)$, we have
\begin{equation}
\left| \int_B f(x)dx -
\sum_{i=1}^n w_i f(\xi_i) \right| < \epsilon \int_B |\sigma(t)|dt.
\end{equation}
\end{lemma}
\begin{proof}
Using \eqref{1300.1}, we know
\begin{equation}\label{1280}
\begin{split}
\left| \int_B f(t) dt - \sum_{i=1}^n w_if(\xi_i) \right|
&=\left| \int_B \int_B \sigma(t) e^{ic \inner{x}{t}} dt dx
- \sum_{i=0}^n w_i \int_B \sigma(t) e^{ic \inner{\xi_i}{t}} dt
\right|\\
&=\left| \int_B \sigma(t) \left( \int_B e^{ic \inner{x}{t}} dx
- \sum_{i=0}^n w_i e^{ic \inner{\xi_i}{t}} \right) dt
\right|.
\end{split}
\end{equation}
Applying (\ref{1300}) to (\ref{1280}), we obtain
\begin{equation}
\begin{split}
\left| \int_B f(t) dt - \sum_{i=1}^n w_if(\xi_i) \right|
&\leq \int_B |\sigma(t)| \left| \int_B e^{ic \inner{x}{t}} dx
- \sum_{i=0}^n w_i e^{ic \inner{\xi_i}{t}} \right| dt\\
& < \epsilon \int_B |\sigma(t)| dt.
\end{split}
\end{equation}
\end{proof}
The following technical lemma bounds the quadrature error when 
integrating complex exponentials with a tensor-product quadrature 
rule that integrates spherical harmonics and certain scaled Bessel functions. 
\begin{lemma}\label{2440}
Let $\rho_1,...,\rho_{m_1} \in [0, 1]$ and $v_1,...,v_{m_1} \in \R$ be the nodes and weights of 
a quadrature rule that integrates
\begin{equation}\label{2800b}
\frac{J_{p/2}(cr\rho)}{(cr\rho)^{p/2}} \rho^{p+1}
\end{equation}
to accuracy $\epsilon > 0$ for all $c >0$ and $r \in [0, 1]$, i.e., 
\begin{align}
\bigg | \int_{0}^1 \frac{J_{p/2}(cr\rho)}{(cr\rho)^{p/2}} \rho^{p+1} d\rho - 
\sum_{i=1}^{m_1} \frac{J_{p/2}(c r \rho_i)}{(cr\rho_i)^{p/2}} \rho_i^{p+1} \bigg| 
< \epsilon.
\end{align}
Let 
$\theta_1,...,\theta_{m_2} \in S_{p+1}$ and $u_1,...,u_{m_2} \in \R$ be the quadrature 
nodes and weights that integrate exactly the spherical harmonics
$S_{N}^{\ell}$ for $N =0, 1, ...,K$ and $\ell=1,...,h(N, p)$. 
Then denoting 
\begin{align}\label{2800f}
t_1,...,t_m \in B \text{ and } w_1,...,w_m \in \R
\end{align}
 the nodes and weights of the quadrature that 
is a tensor product of these two quadratures, we have
for any positive integer $K$ and any integer $p \geq -1$,
\begin{align}\label{2800e}
\bigg| \int_{B} e^{ic\inner{x}{t}} dt - &\sum_{j=1}^m w_j e^{ic \inner{x}{t_j}} \bigg| 
\leq 
 (2\pi)^{p/2+1}
\bigg(\frac{\Gamma(\frac{p}{2}+1)^{1/2}}{\sqrt{2\pi^{p/2+1}}} \epsilon \\
& + \frac{\pi^{p/2 + 1}}{\Gamma(p/2 + 2)}
\sum_{N=K+1}^\infty
 \frac{(c/2)^{N}}{\Gamma(N + p / 2 + 1)} 
\bigg(\frac{(2N + p)(N + p - 1)!}{p!N!}\bigg)^2 
\bigg)
 \end{align}
\end{lemma}

\begin{proof}
Representing a complex exponential as a plane wave expansion, we have
\begin{equation}
\begin{aligned}
& \bigg| \int_{B} e^{ic\inner{x}{t}} dt - \sum_{j=1}^m w_j e^{ic \inner{x}{t_j}} \bigg| =  \\
\bigg | & i^N (2\pi)^{p/2+1} \int_{B} \sum_{N=0}^\infty \sum_{\ell=1}^{h(N,p)}
 \frac{J_{N+p/2}(c\|x\| \|t\|)}{(c\|x\| \|t\|)^{p/2}}
S_N^\ell(x/\|x\|) S_N^\ell(t/\|t\|) dt - \\
& i^N (2\pi)^{p/2+1}
\sum_{N=0}^\infty \sum_{\ell=1}^{h(N,p)}
 S_N^\ell(x/\|x\|) 
 \sum_{j=1}^{m_1}
 \frac{J_{N+p/2}(c\|x\| \rho_j)}{(c\|x\| \rho_j)^{p/2}}
 v_j
\sum_{k=1}^{m_2}
S_N^\ell(\theta_k)  u_k
\bigg|.
\end{aligned}
\end{equation}
Using that $S^{\ell}_N$ integrate to zero for all $N > 0$,
quadrature rule \eqref{2800f} is exact for $S^{\ell}_N$ for $N=0,...,K$, and 
\eqref{2800b} is integrated to $\epsilon$ accuracy,
we obtain
\begin{equation}
\begin{aligned}\label{quad_exp1}
& \bigg| \int_{B} e^{ic\inner{x}{t}} dt - \sum_{j=1}^m w_j e^{ic \inner{x}{t_j}} \bigg|  \\
& \leq 
 (2\pi)^{p/2+1}
\bigg(\frac{\Gamma(\frac{p}{2}+1)^{1/2}}{\sqrt{2\pi^{p/2+1}}} \epsilon + 
\sum_{N=K+1}^\infty \sum_{\ell=1}^{h(N,p)}
 \bigg |  \sum_{j=1}^m 
 \frac{J_{N+p/2}(c\|x\| \| t_j \| )}{(c\|x\| \| t_j \|)^{p/2}}
S_N^\ell(t_j / \| t_j \|)  
 S_N^\ell(x/\|x\|) 
 w_j
\bigg|
\bigg)
\end{aligned}
\end{equation}
Combining \eqref{quad_exp1} with the observation that nodes and weights 
\eqref{2800f} correctly integrate a constant function on $B$, 
we have 
\begin{equation}\label{2800c}
\begin{aligned}
& \bigg| \int_{B} e^{ic\inner{x}{t}} dt - \sum_{j=1}^m w_j e^{ic \inner{x}{t_j}} \bigg|  \\
& \leq 
 (2\pi)^{p/2+1}
\bigg(\frac{\Gamma(\frac{p}{2}+1)^{1/2}}{\sqrt{2\pi^{p/2+1}}} \epsilon + 
\frac{\pi^{p/2 + 1}}{\Gamma(p/2 + 2)}
\sum_{N=K+1}^\infty \sum_{\ell=1}^{h(N,p)}
  \bigg\| \frac{J_{N+p/2}(c\|x\| \| t_j \| )}{(c\|x\| \| t_j \|)^{p/2}} \bigg\|_{\infty}
\|  S_N^\ell(x/\|x\|) \|_{\infty}^2 \bigg).
\end{aligned}
\end{equation}
Applying \eqref{sh_supnorm} and Lemma \ref{1840} to \eqref{2800c},
 we obtain 
 \begin{equation}\label{2800d}
\begin{aligned}
& \bigg| \int_{B} e^{ic\inner{x}{t}} dt - \sum_{j=1}^m w_j e^{ic \inner{x}{t_j}} \bigg|  \\
& \leq
 (2\pi)^{p/2+1}
\bigg(\frac{\Gamma(\frac{p}{2}+1)^{1/2}}{\sqrt{2\pi^{p/2+1}}} \epsilon
+ 
\frac{\pi^{p/2 + 1}}{\Gamma(p/2 + 2)}
\sum_{N=K+1}^\infty \sum_{\ell=1}^{h(N,p)}
 \frac{(c \|x \| \| t \|)^{N }(1 / 2)^{N + p/2}}{\Gamma(N + p / 2 + 1)} 
\frac{(2N + p)(N + p - 1)!}{p!N!} 
\bigg)\\
& \leq
 (2\pi)^{p/2+1}
\bigg(\frac{\Gamma(\frac{p}{2}+1)^{1/2}}{\sqrt{2\pi^{p/2+1}}} \epsilon
+ \frac{\pi^{p/2 + 1}}{\Gamma(p/2 + 2)}
\sum_{N=K+1}^\infty h(N, p)
 \frac{(c/2)^{N}}{\Gamma(N + p / 2 + 1)} 
\frac{(2N + p)(N + p - 1)!}{p!N!}
\bigg)
\end{aligned}
\end{equation}
Inequality \eqref{2800e} follows immediately from substituting \eqref{2400} into \eqref{2800d}.
\end{proof}

\begin{remark}\label{2820}
Lemma \ref{2440} shows that in order to integrate 
a complex exponential on the unit ball, it is sufficient 
to use a quadrature rule that is the tensor product of nodes 
that integrate all spherical harmonics 
$S_N^\ell$ for sufficiently large $N$ and nodes in the radial 
direction that integrate functions of the form 
\begin{equation}\label{2800}
\frac{J_{p/2}(cr\rho)}{(cr\rho)^{p/2}} \rho^{p+1}.
\end{equation}
We will show in Remark \ref{1560} that accurately integrating functions of 
the form of (\ref{2800}) can be done using a quadrature rule
that integrates GPSFs. 
\end{remark}

The following lemma shows that (\ref{2800})
is well represented by an expansion in GPSFs. This lemma
will be used to construct quadrature nodes for integrating 
bandlimited functions. 
\begin{lemma}\label{1440b}
For all real numbers $r,\rho \in (0,1)$,
\begin{equation}\label{1340}
\frac{J_{p/2}(cr\rho)}{(cr\rho)^{p/2}} \rho^{p+1}= 
\sum_{i=0}^{\infty} \beta_{0,i}
\Phi_{0,i}(r)\Phi_{0,i}(\rho)
\end{equation}
where $J_{p/2}$ is a Bessel function, $\Phi_{0,n}$ is
defined in (\ref{5.100}) and $\beta_{0,i}$ is defined in (\ref{7.20}). 
\end{lemma}
\begin{proof}
Since $\Phi_{0,i}$ are complete in $L^2[0,1]_{r^{p+1}}$,
\begin{equation}\label{1360}
\frac{J_{p/2}(cr\rho)}{(cr\rho)^{p/2}} \rho^{p+1} = \sum_{i=0}^\infty 
\sum_{j=0}^{\infty} \alpha_{i,j} \Phi_{0,i}(r)\Phi_{0,j}(\rho)
\end{equation}
where $\alpha_{i,j}$ is defined by the formula
\begin{equation}\label{1320}
\alpha_{i,j} 
= \int_0^1 \int_0^1 \frac{J_{p/2}(cr\rho)}{(cr\rho)^{p/2}} r^{p+1}
\Phi_{0,i}(r) \Phi_{0,j}(\rho) dr \rho^{p+1} d\rho.\\
\end{equation}
Changing the order of integration of (\ref{1320}) and 
substituting in (\ref{7.20}), we obtain,
\begin{equation}\label{1380}
\begin{split}
\alpha_{i,j} 
&= \int_0^1 \Phi_{0,j}(r) \int_0^1 \frac{J_{p/2}(cr\rho)}
{(cr\rho)^{p/2}} \rho^{p+1}
\Phi_{0,i}(\rho) d\rho  r^{p+1} dr\\
&= \beta_{0,i} \int_0^1 \Phi_{0,j}(r)  \Phi_{0,i}(r) r^{p+1} dr \\
&=\delta_{i,j}\beta_{0,i}
\end{split}
\end{equation}
where $\beta_{0,i}$ is defined in (\ref{7.20}). 
Identity (\ref{1340}) follows immediately from the combination of 
(\ref{1360}) and (\ref{1380}). 
\end{proof}
The following remark shows that a quadrature rule that correctly 
integrates certain GPSFs also integrates certain Bessel functions. 
\begin{remark}\label{1560}
Let $\rho_1,...,\rho_n$ be the $n$ roots of $\Phi_{0,n}$ and 
$w_1,...,w_n \in \R$ the $n$ weights such that 
\begin{equation}\label{1260}
\int_0^1 \Phi_{0,k}(r)r^{p+1}dr = \sum_{i=0}^n \Phi_{0,k}(\rho_i)w_i
\end{equation}
for $k=0,1,...,K$. By Lemma \ref{1440b},
\begin{equation}\label{1480}
\begin{split}
&\left| \int_0^1
\frac{J_{p/2}(cr\rho)}{(cr\rho)^{p/2}} \rho^{p+1} d\rho 
- \sum_{i=1}^n \frac{J_{p/2}(cr\rho_i)}{(cr\rho_i)^{p/2}} 
w_i \right|\\
&=
\left| \int_0^1 \left( \sum_{j=0}^{\infty} 
\beta_{0,j} \Phi_{0,j}(r)\Phi_{0,j}(\rho) \right)
\rho^{p+1} d\rho 
- \sum_{i=1}^n w_i
\left( \sum_{j=0}^{\infty} \beta_{0,j} \Phi_{0,j}(r)\Phi_{0,j}(\rho_i)
\right)
\right|
\end{split}
\end{equation}
where $\beta_{0,j}$ is defined in (\ref{7.20}). Applying (\ref{1260})
to (\ref{1480}), we obtain 
\begin{equation}\label{2760}
\begin{split}
&\left| \int_0^1
\frac{J_{p/2}(cr\rho)}{(cr\rho)^{p/2}} \rho^{p+1} d\rho 
- \sum_{i=1}^n \frac{J_{p/2}(cr\rho_i)}{(cr\rho_i)^{p/2}} \rho_i^{p+1}
w_i \right|\\
&=
\left| \int_0^1 \left( \sum_{j=K+1}^{\infty} 
\beta_{0,j} \Phi_{0,j}(r)\Phi_{0,j}(\rho) \right)
 \rho^{p +1} d\rho 
- \sum_{i=1}^n w_i
\left( \sum_{j=K+1}^{\infty} \beta_{0,j} \Phi_{0,j}(r)\Phi_{0,j}(\rho_i)
\right)
\right|.
\end{split}
\end{equation}
As long as $\beta_{0,K+1}$ is in the regime of 
exponential decay, (\ref{2760}) is of magnitude approximately
$\beta_{0,K+1}$. 
\end{remark}

We now describe a quadrature rule that integrates 
complex exponentials. This quadrature rule 
uses nodes that are a tensor product of roots of 
$\Phi_{0,n}$ in the radial direction and nodes that 
integrate spherical harmonics in the angular direction.

\begin{observation}\label{obs2500}
Suppose that $r_1,...,r_n \in (0,1)$ and weights $w_1,...,w_n \in \R$
satisfy
\begin{equation}\label{2500}
\int_0^1 \Phi_{0,k}(r)r^{p+1} dr=\sum_{i=1}^n w_i\Phi_{0,k}(r_i)
\end{equation}
for $k=0,1,...,K_1$. Suppose further that $x_1,...,x_m \in S^{p+1}$ 
are nodes and $v_1,...,v_m \in \R$ are weights such that 
\begin{equation}\label{2520}
\int_{S^{p+1}} S_N^\ell (x) dx = \sum_{i=1}^m v_i S_N^\ell(x_i)
\end{equation}
for all $N \leq K_2$ and for all $\ell \in \{1,2,...,h(N,p)\}$.
Then it follows immediately from Remark \ref{2820} and Remark 
\ref{1560} that 
\begin{equation}\label{2840}
\left| \int_B e^{ic\inner{x}{t}} dt 
- \sum_{i=0}^m v_i \sum_{j=1}^n w_j e^{ic\inner{x}{r_jx_i}}
\right|
\end{equation}
will be exponentially small for large enough $n,m$. 
Lemma \ref{1520} shows that quadrature (\ref{2840}) will also 
accurately integrate functions of the form
\begin{equation}
\int_B \sigma(t) e^{ic\inner{x}{t}} dt
\end{equation}
where $\sigma$ is in $L^2(B)$. 
\end{observation}
\begin{remark}\label{3460}
A generalized Chebyshev quadrature of the form (\ref{2500}) can be generated
by first computing the $n$ roots of $\Phi_{0,n}$ (see Section 
\ref{secroots}) and then solving for $w_1,...,w_n$ the $n \times n$ linear
system of equations
\begin{equation}\label{trans_vander}
\int_0^1 \Phi_{0,k}(r)r^{p+1}dr = \sum_{i=1}^n w_i\Phi_{0,k}(r_i)
\end{equation}
for $k_0, ..., k_{n-1}$ and $r_1,...,r_n$ are the $n$ roots of 
$\Phi_{0,n}$. Linear system \eqref{trans_vander} is the transpose of a 
Vandermonde system. 
If the functions $\Phi_{0, k}$ were polynomials of degree $k$, then this procedure
would construct a Gaussian quadrature, not a Chebyshev quadrature. However, 
the construction of Gaussian quadratures with this method
relies on Euclid’s division algorithm for polynomials, which does not 
hold for the functions $\Phi_{0, 0},...,\Phi_{0, k-1}$. 
Section \ref{secgaussquad} contains a description of an algorithm for 
generating generalized Gaussian quadratures for GPSFs.

\end{remark}
\begin{remark}
The tensor product quadrature rule \eqref{2840} 
uses in a higher concentration of nodes near the origin than near the 
unit sphere. 
An alternative procedure could achieve no loss of accuracy while using 
fewer nodes in the angular direction for smaller radial components. 
This strategy is used in \cite{yoel} for the $p = 0$ GPSFs. Such a 
procedure could be implemented using, for example, methods similar to 
those in \cite{bremer2010}.
\end{remark}
\subsection{Roots of $\Phi_{0,n}$}\label{secroots}
In this section, we describe an algorithm
for finding the roots of $\Phi_{0,n}$.
These roots will be used in the design of quadratures for 
GPSFs. 

The following lemma, which provides a differential equation satisfied
by $\varphi_{N,n}$, follows immediately from the combination of \eqref{7.70}, 
\eqref{7.80}, and Remark \ref{660}. 
It will be used in the evaluation of roots
of $\varphi_{0,n}$ later in this section. 
\begin{lemma}\label{1100}
For all non-negative integers $N, n$,
\begin{equation}\label{1080}
\varphi_{N,n}^{\prime\prime}(r)+\alpha(r)\varphi_{N,n}^{\prime}(r)
+\beta(r)\varphi_{N,n}(r)=0,
\end{equation}
where 
\begin{equation}\label{1120}
\alpha(r)=\frac{-2r}{1-r^2}
\end{equation}
and 
\begin{equation}\label{1140}
\beta(r)=\frac{1/4-(N+p/2)^2}{(1-r^2)r^2}
-\frac{c^2r^2+\chi_{N,n}}{1-r^2}
\end{equation}
where $\varphi_{N,n}$ is defined in (\ref{7.30}) and 
$\chi_{N,n}$ is defined in (\ref{7.70}).
\end{lemma}
The following lemma is obtained by applying the Pr{\"u}fer
Transform (see Lemma \ref{380.2}) to (\ref{1080}). 
\begin{lemma}\label{1160}
For all non-negative integers $n$, real numbers $k>-1$, and $r\in(r_1,r_n)$,
\begin{equation}\label{1180}
\frac{d\theta}{dr}=-\sqrt{\beta(r)}-
\left(\frac{\beta^\prime(r)}{4\beta(r)}+\frac{\alpha(r)}{2}\right)
\sin(2\theta(r)),
\end{equation}
where $r_1, r_n \in (0, 1)$ are the smallest and largest roots of $\varphi_{N, n}$, 
the function $\theta:(0,1) \rightarrow \R$ 
is defined by the formula
\begin{equation}\label{1200}
\frac{\varphi_{N,n}(r)}{\varphi_{N,n}^\prime(r)}=
\sqrt{\beta(r)}\tan(\theta(r)),
\end{equation}
and $\beta^\prime(r)$, the derivative of $\beta(r)$ with respect 
to $r$, is defined by the formula
\begin{equation}\label{1220}
\beta^\prime (r)=
\frac{-2(1/4-(N+p/2)^2)(1-2r^2)}{(1-r^2)r^3}
+\frac{-2rc^2(1-r^2)+2r(-c^2r^2-\chi_{N,n})}{(1-r^2)^2}
\end{equation}
and where $\alpha(r)$, $\beta(r)$ are defined in (\ref{1120})
and (\ref{1140}), $\varphi_{N,n}$ is defined in (\ref{7.30})
and $\chi_{N,n}$ is defined in (\ref{7.70}).
\end{lemma}
\begin{remark}
Extensive numerical results show that for any non-negative integer $n$ and $N = 0$,
$\beta(x)$ is positive and
\begin{equation}
\frac{d\theta}{dr}<0
\end{equation}
for all $r\in(r_1,r_n)$ where $r_1$ and $r_n$ are the smallest 
and largest roots of $\varphi_{N,n}$ respectively. 
Therefore, applying Remark \ref{680.2} to (\ref{1180}), we can view 
$r$ as a function of $\theta$ where $r$ satisfies the differential 
equation
\begin{equation}\label{eq1150}
\frac{dr}{d\theta}=\left( -\sqrt{\beta(r)}-
\left(\frac{\beta^\prime(r)}{4\beta(r)}+\frac{\alpha(r)}{2}\right)
\sin(2\theta(r))
\right)^{-1}
\end{equation}
where $\alpha$, $\beta$, and $\beta^\prime$ are defined in 
(\ref{1120}), (\ref{1140}) and (\ref{1220}).
\end{remark}
The following is a description of an algorithm for the evaluation
of the $n$ roots of $\Phi_{0,n}$. We denote 
the $n$ roots of $\Phi_{0,n}$ by $r_1<r_2<...<r_n$.
\begin{algorithm1}[Find roots of $\Phi_{0, n}$]\,\label{2500alg}
\begin{enumerate}
\setcounter{enumi}{-1}

\item 
Compute the $\tbar_{N,n}$ expansion of $\varphi_{0,n}$ 
using Algorithm \ref{2040}.

\item
Use bisection to find the largest root $r_0\in(0,1)$ of $\beta(r)$
where $\beta(r)$ is defined in (\ref{1140}). If $\beta$ has 
no root on $(0,1)$, then set $r_0=1$.

\item
If $\gamma_{0,n}(c)>1/\sqrt{c}$, place Chebyshev nodes 
(order $5n$, for example) on the
interval $(0,r_0)$ and check, starting at $r_0$ and moving in the 
negative direction, for a sign change. Once a sign change has
occured, use Newton to find an accurate approximation to the root.\\

If $\gamma_{0,n}(c)\leq1/\sqrt{c}$, then use three steps of Muller's
method \cite{muller1956} starting at $r_0$, using the first and second derivatives
of $\varphi_{0,n}$ followed by Newton's method. 

\item 
Defining $\theta_0$ by the formula
\begin{equation}
\theta_0=\theta(r_0),
\end{equation}
where $\theta$ is defined in (\ref{420.2}),
solve the ordinary differential equation $\frac{dr}{d\theta}$
(see (\ref{eq1150})) on the interval $(\pi/2,\theta_0)$, with 
initial condition $r(\theta_0)=r_0$. 
To solve the differential equation, it is sufficient to 
use, for example, second order Runge Kutta with $100$ 
steps (independent of $n$).
We denote by $\tilde{r}_n$ the approximation to $r(\pi/2)$ obtained
by this process. It follows immediately from $(\ref{670.2})$ that
$\tilde{r}_n$ is an approximation to $r_n$, the largest root of 
$\varphi_{N,n}$.

\item\label{step_newton}
Use Newton's method with $\tilde{r}_n$ as an initial guess
to find $r_n$ to high precision. The GPSF $\varphi_{N,n}$ and 
its derivative $\varphi_{N,n}^\prime$ can be evaluated
using the expansion computed in Step 0.

\item 
With initial condition 
\begin{equation}
r(\pi/2)=r_n,
\end{equation}
solve differential equation $\frac{dr}{d\theta}$ (see (\ref{eq1150}))
on the interval 
\begin{equation}
(-\pi/2,\pi/2)
\end{equation}
using, for example, second order Runge Kuta with $100$ steps.
We denote by $\tilde{r}_{n-1}$ the approximation to 
\begin{equation}
r(-\pi/2)
\end{equation}
obtained by this process.

\item
Use Newton's method, with initial guess $\tilde{r}_{n-1}$,
to find to high precision the second largest root, $r_{n-1}$. 

\item
For $k=\{1,2,...,n-1\}$, 
repeat Step 4 on the interval
\begin{equation}
(-\pi/2-k\pi, -\pi/2-(k-1)\pi)
\end{equation}
with intial condition 
\begin{equation}
x(-\pi/2-(k-1)\pi)=r_{n-k+1}
\end{equation}
and repeat Step \ref{step_newton} on $\tilde{r}_{n-k}$.
\end{enumerate}
\end{algorithm1}
\subsection{Generalized Gaussian quadratures for $\Phi_{0,n}$}\label{secgaussquad}
In this section, we describe an algorithm for generating 
generalized Gaussian quadratures for the GPSFs $\Phi_{0,0},\Phi_{0,1},...
,\Phi_{0,2n-1}$ for some integer $n$. 
\begin{definition}\label{1680}
A generalized Gaussian quadrature with respect to a set of functions
\begin{equation}
f_1,...,f_{2n-1}:[a,b]\rightarrow \mathbb{R}
\end{equation}
and non-negative weight function $w:[a,b]\rightarrow \mathbb{R}$ is a set of
$n$ nodes, $x_1,...,x_n\in[a,b]$, and $n$ weights,
$\omega_1,...,\omega_n \in \R$, such that, for any integer $j\leq 2n-1$,
\begin{equation}\label{1700}
\int_a^b f_j(x)w(x)dx=\sum_{i=1}^n\omega_i f_j(x_i).
\end{equation}
A generalized Chebyshev quadrature consists of nodes $x'_1,\dots,x'_{2n-1} \in [a,b]$ and weights $\omega'_1,...,\omega'_{2n-1} \in \R$ such that 
\begin{equation}
\int_a^b f_j(x)w(x)dx=\sum_{i=1}^{2n-1} \omega'_i f_j(x'_i).
\end{equation}
\end{definition}
\begin{remark}
In order to generate a generalized Gaussian quadrature for GPSFs with bandlimit
$c>0$, we first generate a generalized Chebyshev quadrature for GPSFs with 
bandlimit $c/2$ and then, using those nodes and weights as a starting 
point, we use Newton's method with step-length control
to find nodes and weights that form a generalized generalized Gaussian quadrature for 
GPSFs with bandlimit $c$.
\end{remark}
The following is a description of an algorithm for generating 
generalized Gaussian quadratures for the GPSFs
\begin{equation}
\Phi_{0,0}^c,...,\Phi_{0,2n-1}^c.
\end{equation}
\begin{algorithm1}[Generalized Gaussian quadrature for $\Phi_{0,0}^c,...,\Phi_{0,2n-1}^c$]\,\label{alggaussquad}
\begin{enumerate}
\item 
Using Algorithm \ref{2500alg}, generate a generalized Chebyshev quadrature 
for the functions
\begin{equation}
\Phi_{0,0}^{c/2},...,\Phi_{0,n-1}^{c/2}.
\end{equation}
That is, find, $r_1,...,r_n$, the $n$ roots of $\Phi_{0,n}$ 
and weights $w_1,...,w_n$ such that 
\begin{equation}
\int_0^1 \Phi_{0,k}^{c/2}(r) dr 
= \sum_{i=1}^n w_i \Phi_{0,k}^{c/2}(r_i)
\end{equation}
for $k=0,1,...,n-1$. 

\item 
Evaluate the vector $\bm{d}=(d_0,d_1,...,d_{2n-1})$ of discrepencies where 
$d_k$ is defined by the formula
\begin{equation}\label{1580}
d_k=\int_0^1 \Phi_{0,k}^c(r)dr - \sum_{i=1}^n w_i \Phi_{0,k}^c(r_i)
\end{equation}
for $k=0,1,...,2n-1$.

\item
Generate $A$, the $2n \times 2n$ matrix of partial derivatives of $d$
with respect to the $n$ nodes and $n$ weights. Specifically, for $j=0,...,2n-1$, 
the matrix $A$ is defined by the formula
\begin{equation}\label{1600}
A_{i,j} = \left\{
  \begin{array}{ll}
  \Phi_{0,j}^{c}(r_i) & \mbox{for $i = 1,...,n$}, \\[8pt]
  w_i \Phi_{0,j}^{c \prime}(r_i) & \mbox{for $i=n+1,...,2n$}.
  \end{array}
\right.
\end{equation}
where $\Phi_{0,k}^{c \prime}(r)$ denotes the derivative of 
$\Phi_{0,k}^c(r)$ with respect to $r$. \\\\

\item
Solve for $\bm{x} \in \R^{2n}$ the $2n \times 2n$ linear system of equations 
\begin{equation}
A\bm{x}=-\bm{d}
\end{equation}
where $A$ is defined in (\ref{1600}) and $\bm{d}$ is defined in (\ref{1580}).

\item
Update nodes and weights via Newton's method. That is, defining 
$\bm{r} \in \R^{2n}$ to be the vector of nodes and weights
\begin{equation}
\bm{r} = (r_1,r_2,...,r_n,w_1,w_2,...,w_n)^T,
\end{equation}
we construct the updated vector of nodes and weights $\tilde{\bm{r}}$
such that 
\begin{equation}
\tilde{\bm{r}}= \bm{r} + \bm{x}.
\end{equation}

\item
Evaluate $\tilde{\bm{d}}$, the discrepancies \eqref{1580} for the
updated nodes and weights $\tilde{\bm{r}}$. If $\| \tilde{\bm{d}} \|_2 < \| \bm{d} \|_2$, 
continue to the next step. Otherwise, go back to Step 5 and divide the 
steplength by $2$. That is, define $\tilde{\bm{r}}$ by the formula,
\begin{equation}
\tilde{\bm{r}} = \bm{r}+ \frac{1}{2}\bm{x}
\end{equation}
Continue dividing the steplength by $2$ until 
$ \| \tilde{\bm{d}} \|_2 < \| \bm{d} \|_2$. 

\item
Repeat steps 2-6 until the discrepencies, $d_i$ for $i=0,1,...,2n-1$
(see (\ref{1580})) are approximately machine precision. 
\end{enumerate}
\end{algorithm1}
\begin{remark}
It turns out that a generalized Chebyshev quadrature with nodes 
that are roots of $\Phi_{0, n}^{c/2}$ provide a 
reasonable quadrature for $\Phi_{0, 0}^c,...,\Phi_{0, 2n - 1}^c$.
Algorithm \ref{alggaussquad} takes advantage of this fact by using 
this quadrature as a starting point for an optimization that 
finds a Gaussian quadrature for $\Phi_{0, 0}^c,...,\Phi_{0, 2n - 1}^c$.
For the one-dimensional prolate spheroidal wave functions (PSWFs), 
a similar procedure is used to construct quadrature rules \cite{xiao}. 
In \cite{xiao}, the authors show that a generalized Chebyshev quadrature 
for PSWFs with bandlimit $c/2$ (and nodes that are roots of the 
order-$n$ PSWF) integrate to high accuracy the PSWFs of 
bandlimit $c$. This observation relies on a so-called Euclid division 
algorithm for band-limited functions. 
While in the high-dimensional setting we have no analogue of the Euclid 
division theorem, we observe empirically that a similar strategy to the 
one-dimensional approach of \cite{xiao} is a good approximation to a 
Gaussian quadrature. 
\end{remark}
\begin{remark}
The construction of generalized Gaussian quadratures for families of 
functions is now widespread for tasks such as solutions to integral 
equations with singular kernels (see, e.g.,~\cite{ma1996, bremer2010}).
One popular algorithm for generating generalized Gaussian quadratures 
(or near Gaussian) was introduced in \cite{bremer2010}. That method
first constructs a Chebyshev quadrature for some family of 
functions, and then an optimization procedure is used to one-by-one 
remove nodes and readjust the remaining nodes and weights. 
We do not view Algorithm \ref{alggaussquad} as a general tool like that 
of \cite{bremer2010} for constructing generalized Gaussian quadratures. 
Rather, for this particular family of functions, it turns out that we have a
computationally inexpensive, and accurate approximation to a 
generalized Gaussian quadrature that allows for a relatively simple procedure.
\end{remark}

\section{Approximation via GPSF expansions}\label{secinterp}
In this section, we describe a numerical scheme for 
representing a bandlimited function as an expansion in GPSFs.
%

As shown in the context of quadrature (see Section \ref{secquads}),
GPSFs are a natural basis for representing bandlimited functions.
We formulate the approximation problem for GPSFs as recovering the 
coefficients of a bandlimited function $f$ in its GPSF expansion.
That is, suppose that $f$ is of the form
\begin{equation}
f(x)=\int_B \sigma(t) e^{ic\inner{x}{t}} dt
\end{equation}
where $\sigma \in L^2(B)$ for all $x \in B$. Then, $f$ can be represented to arbitrary accuracy
with an expansion of the form 
\begin{equation}
f(x)=\sum_{i=1}^N a_i\psi_i(x)
\end{equation}
where $\psi_j(x)$ is a GPSF defined in (\ref{4.20}) and $a_i$ satisfies
\begin{equation}
a_i=\int_B \psi_i(t) f(t) dt.
\end{equation}
The approximation problem we consider here is the evaluation of the 
coefficients $a_1,...,a_N$. 

The following lemma shows that a quadrature rule that recovers the 
coefficients of the expansion in GPSFs of a complex exponential
will also recover the coefficients in a GPSF expansion of a bandlimited
function. 
\begin{lemma}\label{2080}
Suppose that for all $t\in B$,
\begin{equation}\label{2080.0}
\left| \int_B \psi_j(x) e^{ic\inner{x}{t}} dx - 
\sum_{k=1}^n w_k \psi_j(x_k) e^{ic\inner{x_k}{t}} \right|<\epsilon
\end{equation}
where $B$ denotes the unit ball in $\R^{p+2}$ and $\psi_j$ is 
defined in (\ref{4.20}). Then,
\begin{equation}\label{2080.1}
\left| \int_B \psi_j(x) f(x) dx 
- \sum_{k=1}^n w_k \psi_j(x_k) f(x_k) \right|<\epsilon \int_{B} | \sigma(t) | dt.
\end{equation}
where
\begin{equation}\label{2080.2}
f(x)=\int_B \sigma(t) e^{ic \inner{x}{t}} dt
\end{equation}
for any $\sigma: \R^{p+2} \to \R$.
\end{lemma}
\begin{proof}
Substituting \eqref{2080.2} into \eqref{2080.1} we have
\begin{align}
\bigg| & \int_B \psi_j(x) f(x) dx  - \sum_{k=1}^n w_k \psi_j(x_k) f(x_k) \bigg| \nonumber \\
& = \bigg| \int_{B} \psi_j(x)  \int_B \sigma(t) e^{ic \inner{x}{t}} dt dx - 
\sum_{k=1}^n w_k \psi_j(x_k) \int_B \sigma(t) e^{ic \inner{x_k}{t}} dt
\bigg|\label{2080.3}
\end{align}
Changing the order of integration of \eqref{2080.3}, we obtain
\begin{align}\label{2080.4}
\bigg| & \int_B \psi_j(x) f(x) dx  - \sum_{k=1}^n w_k \psi_j(x_k) f(x_k) \bigg| \nonumber \\
& = \bigg| \int_{B} \sigma(t) \bigg[ \int_{B} \psi_j(x)  e^{ic \inner{x}{t}} dx - 
\sum_{k=1}^n w_k \psi_j(x_k) e^{ic \inner{x_k}{t}} \bigg] dt
\bigg|.
\end{align}
Equation \eqref{2080.1} follows immediately from substituting 
\eqref{2080.0} and \eqref{2080.2} into \eqref{2080.4}. 
\end{proof}
The following lemma shows that the product of a complex exponential 
with a GPSF of bandlimit $c>0$ is a bandlimited function with bandlimit $2c$. 
The proof is a slight modification 
of Lemma 5.3 in \cite{yoel}.
\begin{lemma}\label{2060}
Let $c > 0$ and $\omega \in B$. Then there exists $\sigma: \R^{p+2} \to \R$
such that,
\begin{equation}
\psi_j(x) e^{ic\inner{\omega}{x}}
=\int_B \sigma(\xi) e^{i2c\inner{\xi}{x}}d\xi
\end{equation}
for all $x \in \R^{p+2}$ 
where $\psi_j$ is defined in (\ref{4.20}) and $\sigma$ satisfies
\begin{equation}\label{2560}
\left| \int_B \sigma(t)^2 dt \right| \leq 4/|\lambda_j|^2
\end{equation}
where $\lambda_j$ is defined in (\ref{4.20}). 
\end{lemma}
\begin{proof}
Using (\ref{4.20}),
\begin{equation}\label{3440}
\psi_j(x) e^{ic \inner{\omega}{x}} = \frac{1}{\lambda_j}
\int_B e^{ic\inner{\omega+t}{x}} \psi_j(t) dt.
\end{equation}
Applying the change of variables $\xi=(t+\omega)/2$ to (\ref{3440}),
we obtain
\begin{equation}\label{2540}
\psi_j(x) e^{ic \inner{\omega}{x}} = \frac{1}{\lambda_j}
\int_{B_{\omega}} e^{i2c\inner{\xi}{x}} 2\psi_j(2\xi-\omega) d\xi
\end{equation}
where $B_\omega$ is the ball of radius $1/2$ centered at $\omega/2$. 
It follows immediately from (\ref{2540}) that 
\begin{equation}
\psi_j(x) e^{ic \inner{\omega}{x}} = 
\int_{B_{\omega}} e^{i2c\inner{\xi}{x}} \sigma(\xi) d\xi.
\end{equation}
where
\begin{equation}\label{2580}
\sigma(\xi) = \left\{
  \begin{array}{ll}
  \frac{2\psi_j(2\xi-\omega)}{\lambda_j} & \mbox{if $\xi \in B_\omega$}, \\
  0 & \mbox{otherwise}.
  \end{array}
\right.
\end{equation}
Inequality (\ref{2560}) follows immediately from the combination of 
(\ref{2580}) with the fact that $\psi_j$ is $L^2$ normalized.
\end{proof}
The following observation provides a numerical scheme for recovering
the coefficients in a GPSF expansion of a bandlimited function. 
\begin{observation}\label{obs_quad}
Suppose that $f$ is defined by the formula
\begin{equation}
f(x)=\int_B \sigma(t) e^{ic\inner{x}{t}}dt
\end{equation}
where $\sigma$ is some function in $L^2(B)$. Then, $f$ is representable
in the form 
\begin{equation}\label{3320}
f(x)=\sum_{k=1}^\infty a_k \psi_k(x)
\end{equation}
where 
\begin{equation}
a_k=\int_B f(x) \psi_k(x) dx.
\end{equation}
It follows immediately from the combination of Lemma \ref{2060} and 
Lemma \ref{2080} that using quadrature rule (\ref{2840}) with bandlimit
$2c$ will integrate $a_k$ accurately. That is, following the notation
of Observation \ref{obs2500},
\begin{equation}\label{3300}
\left| a_k -\sum_{i=0}^n w_i \sum_{j=1}^m v_j f(r_ix_j) \psi_k(r_ix_j)
\right|
\end{equation}
is exponentially small for large enough $m,n$.
\end{observation}
\begin{remark}
When integrating a function of the form of (\ref{3320}) 
on the unit disk in $\R^2$, the $v_j$ in (\ref{3300}) are 
defined by the formula
\begin{equation}
v_j=j\frac{2\pi}{2m-1}
\end{equation}
for $j=1,2,...,2m-1$ and the sums
\begin{equation}\label{3300.2}
\sum_{j=1}^m v_j f(r_ix_j) \psi_k(r_ix_j)
\end{equation}
for each $i$ can be computed using an FFT. 
\end{remark}
The following lemma bounds the magnitudes of the coefficients of a
GPSF expansion of a bandlimited function.
\begin{lemma}\label{2860}
Suppose that $f$ is defined by the formula
\begin{equation}
f(x)=\int_B \sigma(t) e^{ic\inner{x}{t}} dt
\end{equation}
for all $x \in B$. Then, 
\begin{equation}\label{2100}
f(x)=\sum_{i=1}^\infty a_i\psi_i(x)
\end{equation}
where $\psi_j(x)$ is a GPSF defined in (\ref{4.20}) and $a_i$ satisfies
\begin{equation}
|a_i| \leq |\lambda_i| \bigg( \int_B |\sigma(t)|^2dt \bigg)^{1/2}
\end{equation}
where $\lambda_i$ is defined in (\ref{4.20}). 
\end{lemma}
\begin{proof}
Since $\psi_j$ form an orthonormal basis for $L^2[B]$, 
$f$ is representable in the form of (\ref{2100}) and for 
all positive integers $i$,
\begin{equation}\label{2120}
a_i=\int_B f(t) \psi_i(t) dt=\int_B \left( \int_B \sigma(\xi) 
e^{ic\inner{t}{\xi}} d\xi \right) \psi_i(t)dt.
\end{equation}
Swapping the order of integration of \eqref{2120} and applying 
\eqref{4.20} and using Cauchy-Schwarz, we obtain
\begin{equation}
\begin{aligned}
|a_i|=|\lambda_i \int_B \sigma(t)\psi_i(t) dt |
& \leq |\lambda_i| \bigg(\int_B |\sigma(t)|^2 dt \int_B |\psi_j(t)|^2 dt \bigg) ^{1/2} \\
& = |\lambda_i| \bigg( \int_B |\sigma(t)|^2dt \bigg)^{1/2}.
\end{aligned}
\end{equation}
\end{proof}
\begin{remark}
Lemma \ref{2860} shows that in order to accurately represent
a bandlimited function, $f$, it is sufficient to find the projection 
of $f$ onto all GPSFs with corresponding eigenvalue above
machine precision. 
\end{remark}
\section{Numerical experiments}\label{secnumres}
The quadrature and approximation formulae described in Sections
\ref{secquads} and \ref{secinterp} were implemented in Fortran 77.  
We used the Lahey/Fujitsu compiler on a 2.9 GHz Intel i7-3520M 
Lenovo laptop. All examples in this section were run in
double precision arithmetic.

In Figures \ref{2160}, \ref{2180}, and \ref{2180_4d}, we plot the 
eigenvalues $|\lambda_{N,n}|$ of integral operator $F_c$ 
(see \eqref{4.10}) for different $N$ and different $c$ for 
$p = 0, 1, 2$ , i.e., GPSFs defined on the unit ball in $\R^{p+2}$. 
We note that for fixed $c, p$, as $N$ increases, there are fewer eigenvalues 
above $\epsilon$ for any $0 < \epsilon < 1$. This behavior is expected, but an 
analysis of the spectrum of $F_c$ is beyond the scope of this paper. 

Figures \ref{2200}, \ref{2230}, \ref{2240}, \ref{2260}, and \ref{2260_4d}
include plots of the GPSFs $\Phi_{N,n}(r)$ (see (\ref{5.100})) 
for different $N,n$, $c$, and for $p=0, 1, 2$. 
In Figure \ref{2200} we plot $\Phi_{0, n}$ for bandlimit $c = 50$ and 
$p = 1$ for $n = 0, 2, 5, 10$. 
Figure \ref{2230} includes plots of $\Phi_{0, n}$ for $c = 100, p=0$, 
and $n = 0, 2, 5, 10$.
We plot $\Phi_{10, n}$ in Figure \ref{2240} for $c= 50, p=1$, and $n=0, 2, 5, 10$. 
In Figure \ref{2260} we plot $\Phi_{25, n}$ for $c = 100, p = 0$, and $n = 0, 2, 5, 10$. 
Lastly, Figure \ref{2260_4d} includes plots of $\Phi_{25, n}$ for $c=200, p = 2$ for 
$n = 0, 2, 5, 10$. 
 
Figure \ref{7860s} plots the 
magnitudes of coefficients of the GPSF expansion of the complex 
exponential $e^{ic\inner{x}{t}}$ for all $t$ on the unit disk where $x=(0.3,0.4)$
and $c=50$.
The vertical axis, $|\alpha_{N,n}|$, is the magnitude of the coefficient of 
$\Phi_{N,n}(r)\sin(\theta)$ in the GPSF expansion of (\ref{3480}). 
These coefficients were obtained via formula (\ref{3300}). Their exponential 
decay is explained in Observation \ref{obs_quad}. 
Figure \ref{7860s} demonstrates that for smaller $N$, there are more coefficients 
above machine precision. This behavior is related to the distribution of 
the eigenvalues in Figures \ref{2160}, \ref{2180}, and \ref{2180_4d}. 


In Tables \ref{7600}-\ref{7820}, we provide the 
accuracy of quadrature rule (\ref{2840}) 
in integrating the function 
\begin{equation}\label{3480}
e^{ic\inner{x}{t}}
\end{equation}
over the unit disk where $x=(0.9,0.2)$. We provide results
for $c=20$ and $c=100$ using generalized Chebyshev and 
generalized Gaussian quadratures in the radial direction.
Generalized Chebyshev quadratures were generated
 via the method described in Remark \ref{3460}. 
 The generalized Gaussian quadratures were constructed with Algorithm \ref{alggaussquad}. 
In each table in this section, 
the column labeled ``$c$'' denotes the bandlimit $c$ in (\ref{3480}).
Relative errors of quadrature are denoted ``relative error''
and the true value of the integral 
was obtained by a calculation in extended precision using a large 
number of nodes. 

  In all examples in this section, Algorithm \ref{alggaussquad} for generating 
  generalized Gaussian quadratures converged to full 
 double precision accuracy in several seconds and within 5 iterations. For example,
 for $c =100$ and $p = 0$, Algorithm \ref{alggaussquad} took $6$ seconds 
 and $4$ iterations for a discrepancy $\| \bm{d} \| < 10^{-14}$ (see \eqref{1580}). 

In Tables \ref{7600}, \ref{7620}, and \ref{7640} we set $c = 20$. 
Tables \ref{7600} and \ref{7620} use the generalized Chebysev quadrature
in the radial direction. 
In Table \ref{7600} we fix the number of nodes in the angular 
direction at $50$ and demonstrate convergence in the number of nodes in the
radial direction. Table \ref{7620} fixes the number of radial nodes at $14$ and 
demonstrates convergence in the number of angular nodes. 
In Table \ref{7640} we demonstrate convergence of the generalized 
Gaussian quadrature in the radial direction with $50$ angular nodes. 

In Tables \ref{7800}, \ref{7810}, and \ref{7820} we set $c = 100$. 
Tables \ref{7800} and \ref{7810} use the generalized Chebysev 
quadrature in the radial direction. 
In Table \ref{7800} we fix the number of nodes in the angular 
direction at $140$ and demonstrate convergence in the number of nodes in the 
radial direction of the generalized Chebyshev quadrature. Table \ref{7810} fixes 
the number of radial nodes at $40$ and 
demonstrates convergence in the number of angular nodes. 
In Table \ref{7820} we demonstrate the accuracy of the generalized 
Gaussian quadrature with $150$ angular nodes and a varying number of nodes
in the radial direction. 

\begin{figure}[htb!]
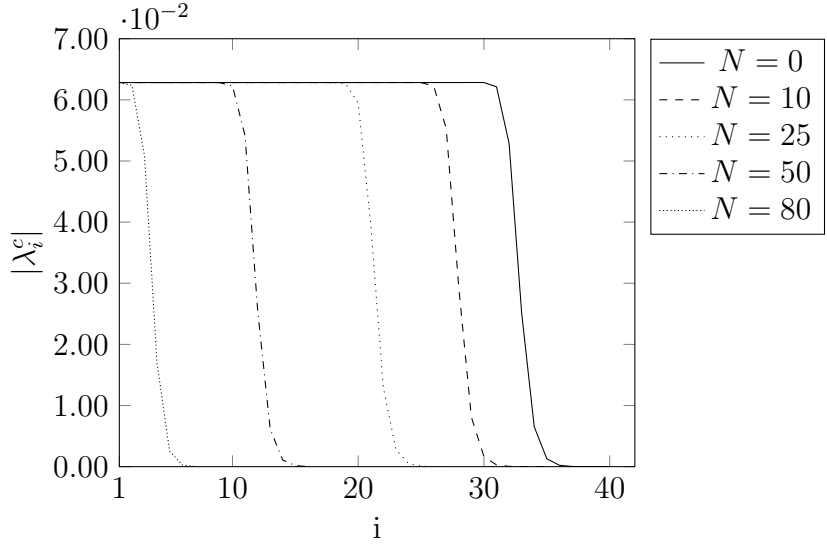

\centering


\caption{
Eigenvalues of $F_c$ (see (\ref{4.10})) for $c=100$ and 
$p=0$
}
\label{2160}
\end{figure}
\begin{figure}[ht!]
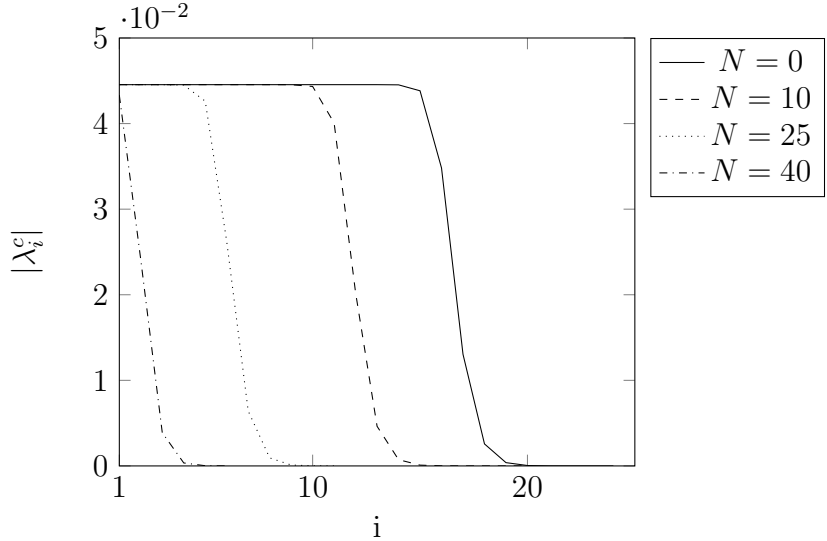

\centering
%

\caption{
Eigenvalues of $F_c$ (see (\ref{4.10})) for $c=50$ and 
$p=1$
}
\label{2180}
\end{figure}
\begin{figure}[ht!]
\centering
%

\caption{
Eigenvalues of $F_c$ (see (\ref{4.10})) for $c=200$ and 
$p=2$
}
\label{2180_4d}
\end{figure}
\begin{figure}[ht!]
\centering
%

\caption{
Plots of GPSFs $\Phi_{0,n}$ (see (\ref{5.100})) with $c=50$ and $p=1$
}
\label{2200}
\end{figure}
\begin{figure}[ht!]
\centering
%

\caption{
Plots of GPSFs $\Phi_{0,n}$ (see (\ref{5.100})) with $c=100$ and $p=0$
}
\label{2230}
\end{figure}
\begin{figure}[ht!]
\centering
%

\caption{
Plots of GPSFs $\Phi_{10,n}$ (see (\ref{5.100})) with $c=50$ and $p=1$
}
\label{2240}
\end{figure}
\begin{figure}[ht!]
\centering
%

\caption{
Plots of GPSFs $\Phi_{25,n}$ (see (\ref{5.100})) with $c=200$ and $p=2$
}
\label{2260_4d}
\end{figure}

\begin{table}[H]

  \centering
\begin{tabular}{cccl}
  $c$ & radial nodes & angular nodes & relative error \\
  \hline
 $ 20$&$  6$&$ 50$&$ 0.84109\times 10^{0}$\\
&$  8$&&$ 0.70864\times 10^{-3}$\\
&$ 10$&&$ 0.15834\times 10^{-7}$\\
 &$ 12$&&$ 0.75601\times 10^{-13}$\\
 &$ 14$&&$ 0.68485\times 10^{-14}$\\
 &$ 16$&&$ 0.29262\times 10^{-14}$\\
 &$ 18$&&$ 0.75991\times 10^{-14}$\\
  \hline
\end{tabular}

  \caption{Quadratures for $e^{ic\inner{x}{t}}$ where $x=(0.9,0.2)$
  over the unit disk using several different numbers of radial nodes 
  for $c=20$.
  Generalized Chebyshev quadratures are used in the radial direction.}
    \label{7600}
\end{table}
%
%
%
%
\begin{table}[H]

  \centering
\begin{tabular}{cccl}
  $c$ & radial nodes & angular nodes & relative error \\
  \hline
 $ 20$&$ 14$&$ 20$&$ 0.46437\times 10^{0}$\\
 &&$ 25$&$ 0.18500\times 10^{-1}$\\
 &&$ 30$&$ 0.14547\times 10^{-3}$\\
 &&$ 35$&$ 0.64949\times 10^{-7}$\\
 &&$ 40$&$ 0.25015\times 10^{-9}$\\
 &&$ 45$&$ 0.16653\times 10^{-12}$\\
 &&$ 50$&$ 0.51483\times 10^{-14}$\\
 &&$ 55$&$ 0.30672\times 10^{-14}$\\
 &&$ 60$&$ 0.53592\times 10^{-14}$\\
  \hline
\end{tabular}

  \caption{Quadratures for $e^{ic\inner{x}{t}}$ where $x=(0.9,0.2)$
  over the unit disk using several different numbers of angular nodes for $c=20$.
  Generalized Chebyshev quadratures are used in the radial direction.}
    \label{7620}
\end{table}
%
%
%
%
\begin{table}[H]

  \centering
\begin{tabular}{cccl}
  $c$ & radial nodes & angular nodes & relative error \\
  \hline
 $ 20$&$  4$&$ 50$&$ 0.12603\times 10^{0}$\\
 &$  6$&&$ 0.36513\times 10^{-6}$\\
 &$  8$&&$ 0.41931\times 10^{-12}$\\
 &$ 10$&&$ 0.15463\times 10^{-14}$\\
 &$ 12$&&$ 0.35160\times 10^{-14}$\\
  \hline
\end{tabular}

  \caption{Quadratures for $e^{ic\inner{x}{t}}$ where $x=(0.9,0.2)$
  over the unit disk using several different numbers of radial nodes for $c=20$.
  Generalized Gaussian quadratures generated via Algorithm \ref{alggaussquad} are 
  used in the radial direction.}
    \label{7640}
\end{table}
%
%
%
%
\begin{table}[H]

  \centering
\begin{tabular}{cccl}
  $c$ & radial nodes & angular nodes & relative error \\
  \hline
 $100$&$ 30$&$140$&$ 0.10612\times 10^{2}$\\
 &$ 32$&&$ 0.11305\times 10^{0}$\\
 &$ 34$&&$ 0.45510\times 10^{-4}$\\
 &$ 36$&&$ 0.63672\times 10^{-6}$\\
 &$ 38$&&$ 0.54009\times 10^{-9}$\\
 &$ 40$&&$ 0.94943\times 10^{-13}$\\
  \hline
\end{tabular}
  \caption{Quadratures for $e^{ic\inner{x}{t}}$ where $x=(0.9,0.2)$
  over the unit disk using several different numbers of radial nodes for $c=100$. Chebyshev
  quadratures are used in the radial direction.}
    \label{7800}
\end{table}
%
%
%
%
\begin{table}[H]

  \centering
\begin{tabular}{cccl}
  $c$ & radial nodes & angular nodes & relative error \\
  \hline
 $100$&$ 40$&$115$&$ 0.12341\times 10^{-3}$\\
&&$120$&$ 0.12633\times 10^{-5}$\\
&&$125$&$ 0.28112\times 10^{-7}$\\
&&$130$&$ 0.60096\times 10^{-9}$\\
&&$135$&$ 0.13296\times 10^{-11}$\\
&&$140$&$ 0.94943\times 10^{-13}$\\
&&$145$&$ 0.23749\times 10^{-12}$\\
&&$150$&$ 0.16075\times 10^{-12}$\\
  \hline
\end{tabular}
  \caption{Quadratures for $e^{ic\inner{x}{t}}$ where $x=(0.9,0.2)$
  over the unit disk using several different numbers of angular nodes for $c=100$. Chebyshev
  quadratures are used in the radial direction.}
    \label{7810}
\end{table}
%
%
%
%
\begin{table}[H]
  \centering
\begin{tabular}{cccl}
  $c$ & radial nodes & angular nodes & relative error \\
  \hline
 $100$&$ 20$&$150$&$ 0.77025\times 10^{-5}$\\
 &$ 22$&&$ 0.20280\times 10^{-9}$\\
 &$ 24$&&$ 0.28465\times 10^{-12}$\\
 &$ 26$&&$ 0.50904\times 10^{-13}$\\
 &$ 28$&&$ 0.35430\times 10^{-12}$\\
 &$ 30$&&$ 0.39846\times 10^{-12}$\\
  \hline
\end{tabular}
  \caption{Quadratures for $e^{ic\inner{x}{t}}$ where $x=(0.9,0.2)$
  over the unit disk using several different numbers of radial nodes for $c=100$. Gaussian
  quadratures generated using Algorithm \ref{alggaussquad} are used in the radial 
  direction.}
    \label{7820}
\end{table}
\begin{figure}[htb!]
\centering
%

\caption{
Coefficients, obtained via formula (\ref{3300}), of the GPSF
 expansion of the function on the unit disk
 $e^{ic\inner{x}{t}}$ where $x=(0.3,0.4)$ where $c=50$. 
 In the radial direction, $40$ nodes were used and $140$ nodes 
 were used in the angular direction.
}
\label{7860s}
\end{figure}

\section{Acknowledgements}
This work is the culmination of many discussions with Kirill Serkh.
A preliminary version is in a joint unpublished technical report \cite{greengard2018}. 
This paper is in some sense the completion of work started there.
We thank Kirill Serkh for his help.  
The author thanks Vladimir Rokhlin, Manas Rachh, and Jeremy Hoskins
for many useful discussions.

\section{Appendix A}
This appendix includes several technical results that were used in the 
main sections of this paper. 

\subsection{Properties of the derivatives of GPSFs}

%

The following theorem follows immediately from~(\ref{7.30})
and~(\ref{7.70}).
%
\begin{theorem} \label{thm12.1}
Let $c>0$. Then
  \begin{align}
&\hspace*{-2em} \frac{d}{dx} \Bigl( (x^{p+1}-x^{p+3}) 
  \frac{d\Phi_{N,n}}{dx}(x) \Bigr) \notag \\
&+ \Bigl( \chi_{N,n}x^{p+1} - \frac{(p+1)(p+3)}{4}x^{p+1}
- N(N+p)x^{p-1} - c^2x^{p+3} \Bigr)\Phi_{N,n}(x) =0,
    \label{12.10}
  \end{align}
where $0<x<1$ and $N$ and $n$ are arbitrary nonnegative integers.
\end{theorem}
\begin{corollary} \label{cor12.2}
Let $c>0$. Then
  \begin{align}
&\hspace*{-2em} x^{2}(1-x^2)\Phi_{N,n}''(x) 
  + \bigl( (p+1)x - (p+3) x^{3} \bigr)\Phi_{N,n}'(x) \notag \\
&\hspace*{-0em} + \Bigl( \chi_{N,n}x^{2} - \frac{(p+1)(p+3)}{4}x^{2} 
- N(N+p) - c^2x^{4}
\Bigr)\Phi_{N,n}(x) = 0,
    \label{12.20}
  \end{align}
where $0<x<1$ and $N$ and $n$ are arbitrary nonnegative integers.
\end{corollary}

The following lemma connects the values of the $(k+2)$nd derivative of the
function $\Phi_{N,n}$ with its derivatives of orders $k-4, k-3, \ldots, k+1$,
and is obtained by repeated differentiation of~(\ref{12.20}).
\begin{lemma} \label{lem12.3}
Let $c > 0$. Then
  \begin{multline}
\hspace*{-4em} (x^2-x^4)\Phi_{N,n}^{(k+2)}(x) 
+\bigl((2k+1+p)x-(4k+3+p)x^3\bigr)\Phi_{N,n}^{(k+1)}(x) \\
\hspace*{-2em} +\Bigl( k(k+p)-N(N+p) + \bigl[ \chi_{N,n} 
  -\tfrac{1}{4}(p+1)(p+3) \\
\hspace{0em} -3k(2k+1+p) \bigr]x^2 - c^2x^4 \Bigr)\Phi_{N,n}^{(k)}(x) \\
\hspace{-2em}+\Bigl( \bigl[ 2k\bigl(\chi_{N,n} -\tfrac{1}{4}(p+1)(p+3) \bigr)
  - k(k-1)(4k+1+3p) \bigr]x -4kc^2x^3 \Bigr)\Phi_{N,n}^{(k-1)}(x) \\
\hspace{-2em}+\Bigl( k(k-1)\bigl(\chi_{N,n} -\tfrac{1}{4}(p+1)(p+3) \bigr)
  - k(k-1)(k-2)(k+p) - 6k(k-1)c^2x^2 \Bigr)\Phi_{N,n}^{(k-2)}(x) \\
\hspace{-2em}-4k(k-1)(k-2)c^2x\Phi_{N,n}^{(k-3)}(x) 
-k(k-1)(k-2)(k-3)c^2\Phi_{N,n}^{(k-4)}(x) = 0,
    \label{12.30}
  \end{multline}
where $0<x<1$, $N$ and $n$ are arbitrary nonnegative integers, and $k$ is an
arbitrary integer such that $k\ge 4$. Also,
  \begin{multline}
\hspace*{-4em} (x^2-x^4)\Phi_{N,n}''(x) 
+ \bigl( (p+1)x-(p+3)x^3 \bigr)\Phi_{N,n}'(x) \\
+ \Bigl( -N(N+p) +\bigl[\chi_{N,n} - \tfrac{1}{4}(p+1)(p+2)\bigr]x^2
  -c^2x^4 \Bigr)\Phi_{N,n}(x) = 0,
    \label{12.40}
  \end{multline}
and
  \begin{multline}
\hspace*{-4em} (x^2-x^4)\Phi_{N,n}^{(3)}(x)
+ \bigl( (p+3)x-(p+7)x^3 \bigr)\Phi_{N,n}''(x) \\
+ \Bigl( (p+1)-N(N+p) + \bigl[\chi_{N,n} - \tfrac{1}{4}(p+1)(p+3) - 3(p+3)
  \bigr]x^2 - c^2x^4 \Bigr)\Phi_{N,n}'(x) \\
+ \Bigl( 2 \bigl[\chi_{N,n} - \tfrac{1}{4}(p+1)(p+3)\bigr] x - 4c^2x^3 \Bigr)
  \Phi_{N,n}(x) = 0,
    \label{12.50}
  \end{multline}
and
  \begin{multline}
\hspace*{-4em} (x^2-x^4)\Phi_{N,n}^{(4)}(x) 
+ \bigl((p+5)x-(p+11)x^3\bigr)\Phi_{N,n}^{(3)}(x) \\
+ \Bigl( 2(p+2) - N(N+p) +\bigl[ \chi_{N,n}-\tfrac{1}{4}(p+1)(p+3) 
  - 6(p+5) \bigr]x^2 - c^2x^4 \Bigr)\Phi_{N,n}''(x) \\
+ \Bigl( \bigl[ 4\bigl( \chi_{N,n}-\tfrac{1}{4}(p+1)(p+3)\bigr) - 6(p+3)
  \bigr]x - 8c^2x^3 \Bigr)\Phi_{N,n}'(x) \\
+ \Bigl( 2\bigl(\chi_{N,n}-\tfrac{1}{4}(p+1)(p+3) \bigr) - 12c^2x^2 \Bigr)
  \Phi_{N,n}(x) = 0,
    \label{12.60}
  \end{multline}
and
  \begin{multline}
\hspace*{-4em} (x^2-x^4)\Phi_{N,n}^{(5)}(x)
+ \bigl( (p+7)x-(p+15)x^3 \bigr) \Phi_{N,n}^{(4)}(x) \\
+ \Bigl( 3(p+3)-N(N+p) + \bigl[ \chi_{N,n} -\tfrac{1}{4}(p+1)(p+3) 
  - 9(p+7) \bigr]x^2 - c^2x^4 \Bigr)\Phi_{N,n}^{(3)}(x) \\
+ \Bigl( \bigl[ 6\bigl( \chi_{N,n}-\tfrac{1}{4}(p+1)(p+3) \bigr)
  - 6(3p+13) \bigr]x - 12c^2x^3 \Bigr)\Phi_{N,n}''(x) \\
+ \Bigl( 6 \bigl( \chi_{N,n}-\tfrac{1}{4}(p+1)(p+3) \bigr) -6(p+3)
  - 36c^2x^2 \Bigr)\Phi_{N,n}'(x) \\
- 24c^2x\Phi_{N,n}(x) = 0,
    \label{12.70}
  \end{multline}
where $0<x<1$ and $N$ and $n$ are arbitrary nonnegative integers.
\end{lemma}

The following corollary and theorem are obtained immediately from
Lemma~\ref{lem12.3}.
\begin{corollary} \label{cor12.4}
Let $c > 0$. Then
  \begin{multline}
\hspace*{-3em} \bigl( k(k+p)-N(N+p) \bigr)\Phi_{N,n}^{(k)}(0) \\
+ \Bigl( k(k-1)\bigl( \chi_{N,n} -\tfrac{1}{4}(p+1)(p+3) \bigr)
  - k(k-1)(k-2)(k+p) \Bigr) \Phi_{N,n}^{(k-2)}(0) \\
- k(k-1)(k-2)(k-3)c^2 \Phi_{N,n}^{(k-4)}(0) = 0,
    \label{12.80}
  \end{multline}
where $N$ and $n$ are arbitrary nonnegative integers, and $k$ is an
arbitrary integer so that $k\ge 4$. Also,
  \begin{align}
N(N+p)\Phi_{N,n}(0) = 0,
    \label{12.90}
  \end{align}
and
  \begin{align}
\bigl( (p+1)-N(N+p) \bigr)\Phi_{N,n}'(0) = 0,
    \label{12.100}
  \end{align}
and
  \begin{align}
\bigl( 2(p+2)-N(N+p) \bigr)\Phi_{N,n}''(0) 
+ 2\bigl( \chi_{N,n} -\tfrac{1}{4}(p+1)(p+3)\bigr) \Phi_{N,n}(0) = 0,
    \label{12.110}
  \end{align}
and
  \begin{multline}
\bigl( 3(p+3)-N(N+p) \bigr) \Phi_{N,n}^{(3)}(0) \\
+ \Bigl( 6\bigl(\chi_{N,n} -\tfrac{1}{4}(p+1)(p+3) \bigr) - 6(p+3) \Bigr)
  \Phi_{N,n}'(0) = 0,
    \label{12.120}
  \end{multline}
where $N$ and $n$ are arbitrary nonnegative integers.
\end{corollary}
%
%
The following theorem follows immediately from \eqref{940}.
\begin{theorem} \label{thm12.55}
For all integers $N \geq 1$, 
%
%
%
  \begin{equation}
\Phi_{N,n}^{(k)}(0) = 0 \quad \mbox{for $k=0,1,\ldots,N-1$},
    \label{12.140}
  \end{equation}
  where $n$ is an arbitrary nonnegative integer.
\end{theorem}

\begin{theorem} \label{thm12.6}
Suppose that $N$ and $n$ are nonnegative integers. Then
\begin{align}
\Phi_{N,n}(1)\ne 0.
\end{align}
\end{theorem}
\begin{proof}
Suppose that $\Phi_{N, n}(1) = 0$. Then using Lemma \ref{lem12.3},
we know $\Phi_{N, n}^{(k)}(1) = 0$ for all non-negative $k$. Since $\Phi_{N, n}$ 
is analytic in the complex plane, we have $\Phi_{N, n}(x) = 0$
for all $x \in \R$. 
\end{proof}

%
%
%
%
%

\subsection{Derivation of the integral operator $Q_c$} \label{sect9}

In this section we derive an explicit formula for the integral
operator $Q_c$, defined in~(\ref{4.30}). 

Suppose that $B$ denotes the closed unit
ball in $\R^{p+2}$. From~(\ref{4.30}), we have
  \begin{align}
&Q_c[\psi](x) 
= \Bigl(\frac{c}{2\pi}\Bigr)^{p+2}
  \int_B \int_B e^{ic\inner{x-t}{u}}
   \psi(t) \, du \, dt,
   \label{9.10}
  \end{align}
for all $x\in B$.
It follows from Formula 10.9.4 of \cite{dlmf} that
\begin{align}\label{10.9.4_dlmf}
  \int_{S^{p+1}} e^{ir \inner{v}{u}} du = (2\pi)^{p/2+1} J_{p/2}(r \| v \|)/(r \| v \|)^{p/2}
\end{align}
for all $v \in B$ and $r \in [0, 1]$ where $J_\nu$
denotes Bessel functions of the first kind (see Section~VII
of~\cite{bell4}). 
Multiplying both sides of \eqref{10.9.4_dlmf} by $r^{p+1}$ and integrating 
with respect to $r$ we obtain
%
\begin{align}\label{9.49}
\int_B e^{ic\inner{v}{u}}\, du &= (2\pi)^{p/2+1}
  \int_0^1 \frac{J_{p/2}(c\|v\|\rho)}{(c\|v\|\rho)^{p/2}} \rho^{p+1}\, d\rho.
\end{align}
It follows immediately from the combination of \eqref{9.49} and formula~6.561(5)
in~\cite{gradshteyn} that 
  \begin{align} \label{9.50}
\int_B e^{ic\inner{v}{u}}\, du = \bigg( \frac{2\pi}{c} \bigg)^{p/2+1}
  \frac{J_{p/2+1}(c\|v\|)}{\|v\|^{p/2+1}},
  \end{align}
for all $v \in B$. Combining~(\ref{9.10}) and~(\ref{9.50}),
  \begin{align}
&Q_c[\psi](x) 
= \bigg(\frac{c}{2\pi} \bigg)^{p/2+1}
  \int_B \frac{J_{p/2+1}
  \bigl(c\|x-t\|\bigr)}{\|x-t\|^{p/2+1}} \psi(t) \, dt,
  \end{align}
for all $x\in B$.
%
%
%

%

\bibliographystyle{abbrv}
\bibliography{refs}

\end{document}